\documentclass[a4paper,11pt]{amsart}
 \usepackage[english]{babel}
  \usepackage{amsmath,amsfonts,amssymb,amsthm,relsize,setspace,nicefrac, yhmath, amscd,eucal}
  \usepackage{amsrefs, mathrsfs}
 \usepackage[colorlinks, linkcolor=red, citecolor=blue, urlcolor=blue, hypertexnames=true]{hyperref}

  \usepackage{hyperref} 

 \usepackage[active]{srcltx} 
 \usepackage{color,soul} 
\newtheorem{theorem}{Theorem}[section]

\newtheorem{lemma}[theorem]{Lemma}
\newtheorem{proposition}[theorem]{Proposition}

\theoremstyle{definition}
\newtheorem{definition}[theorem]{Definition}
\newtheorem{remark}{Remark}

\newcommand{\be}{\begin{equation}}
\newcommand{\bel}[1]{\begin{equation}\label{#1}}
\newcommand{\ee}{\end{equation}}

\newcommand{\barr}{\begin{eqnarray}}
\newcommand{\earr}{\end{eqnarray}}
\newcommand{\bars}{\begin{eqnarray*}}
\newcommand{\ears}{\end{eqnarray*}}

\newtheorem{subn}{\name}

\newcommand{\bsn}[1]{\def\name{#1}\begin{subn}}
\newcommand{\esn}{\end{subn}}

\newtheorem{sub}{\name}[section]

\newcommand{\bs}{\begin{sub}}
\newcommand{\es}{\end{sub}}

\newcommand{\bth}[1]{\def\name{Theorem}
\begin{sub}\label{t:#1}}
\newcommand{\blemma}[1]{\def\name{Lemma}
\begin{sub}\label{l:#1}}
\newcommand{\bcor}[1]{\def\name{Corollary}
\begin{sub}\label{c:#1}}
\newcommand{\bdef}[1]{\def\name{Definition}
\begin{sub}\label{d:#1}}
\newcommand{\bprop}[1]{\def\name{Proposition}
\begin{sub}\label{p:#1}}

\newcommand{\BA}{\begin{array}}
\newcommand{\EA}{\end{array}}
\newcommand{\BAN}{\renewcommand{\arraystretch}{1.2}
\setlength{\arraycolsep}{2pt}\begin{array}}
\newcommand{\BAV}[2]{\renewcommand{\arraystretch}{#1}
\setlength{\arraycolsep}{#2}\begin{array}}
\newcommand{\BSA}{\begin{subarray}}
\newcommand{\ESA}{\end{subarray}}

\newcommand{\BAL}{\begin{aligned}}
\newcommand{\EAL}{\end{aligned}}
\newcommand{\BALG}{\begin{alignat}}
\newcommand{\EALG}{\end{alignat}}
\newcommand{\BALGN}{\begin{alignat*}}
\newcommand{\EALGN}{\end{alignat*}}
\def\angb<#1>{\langle #1 \rangle}

\def\Gw{\Omega}              

   \def\CO{{\mathcal O}}   
   \def\CB{{\mathcal B}}   \def\CC{{\mathcal C}}
\def\CD{{\mathcal D}}

\def\CW{{\mathcal W}} 

\def\N{\mathbb{N}}

\def\R{\mathbb{R}}

\let\.=\cdot
\let\0=\emptyset

\def\O{\Omega}

\numberwithin{equation}{section}

\theoremstyle{definition}

\let\.=\cdot
\let\0=\emptyset

\def\O{\Omega}

\newenvironment{formula}[1]{\begin{equation}\label{eq:#1}}
                       {\end{equation}\noindent}

\def\Fi#1{\begin{formula}{#1}}
\def\Ff{\end{formula}\noindent}

\setstcolor{red}

\begin{document}


 \title[]{Principal eigenvalue and positive solutions for Fractional $P-Q$ Laplace operator in quantum field theory}

\author{Thanh-Hieu Nguyen}
\address{Department of Mathematics and Applications, Saigon University, 273 An Duong Vuong st., Ward 3, Dist.5, Ho Chi Minh City, Viet Nam}
\email{thanhhieukhtn@gmail.com}

\author{Hoang-Hung Vo$^{*}$}
\address{Department of Mathematics and Applications, Saigon University, 273 An Duong Vuong st., Ward 3, Dist.5, Ho Chi Minh City, Viet Nam}
\email{vhhungkhtn@gmail.com}
\thanks{$^*$ Corresponding author}


\begin{abstract}
This article  deals with the existence and non-existence 
of  positive solutions for the eigenvalue problem driven by nonhomogeneous
fractional $p\& q$ Laplacian operator  with indefinite weights
$$\left(-\Delta_p\right)^{\alpha}u + \left(-\Delta_q\right)^{\beta}u \,= \lambda\left[a(x) \left|u\right|^{p-2}u + b(x) \left|u\right|^{q-2}u \right]\quad\quad\textrm{in $\O$},$$
where $\O$ is a smooth bounded domain in $\R^N$ extended by zero outside. When $\O=\R^N$ and $b\equiv0$, we further show that there exists a continuous family of the eigenvalue if  $1<q<p<q^*_\beta=\frac{Nq}{N-q\beta}$ and  $0\leq a\in L^{\left(\frac{q_{\beta}^*}{s}\right)'}\left(\R^N\right)\bigcap L^{\infty}\left(\R^N\right)$ with $s$ satisfies $\dfrac{p-t}{p_{\alpha}^*}+ \dfrac{p\left(1-t\right)}{s} =1$, for some $t\in \left(0, \sqrt{\dfrac{p-q}{p}}\right).$
Our approach replies strongly on variational analysis, in which the Mountain pass theorem plays the key role.  The main difficulty in this study is that how to establish the Palais-Smale conditions. In particular, in $\R^N$, due to the lack of spatial compactness and the embedding  $W^{\alpha, p}\left(\R^N\right) \hookrightarrow W^{\beta, q}\left(\R^N\right)$, we must employ the concentration-compactness principle of P.L. Lions \cite{PLL} to overcome the difficulty.
\end{abstract}


\date{\today}



\maketitle

\tableofcontents

\section{\bf Introduction and main results}
In this paper, we investigate the existence and non-existence of positive solution of the following fractional 
$\left(p,q\right)$-Laplacians :
\begin{align}\label{1.1}
\left\{\begin{array}{ll}
\left(-\Delta_p\right)^{\alpha}u + \left(-\Delta_q\right)^{\beta}u \,= \lambda\left[a(x) \left|u\right|^{p-2}u + b(x) \left|u\right|^{q-2}u \right] &\mbox{in } \Omega, \smallskip\\
u=0&\mbox{in } \mathbb{R}^N \setminus \Omega,
\end{array}\right.
\end{align}
where $\lambda \in\R$,\, $0<\alpha<\beta<1<q\leq p<\infty$; 
 $(-\Delta)_r^{s}$ is the regional fractional p-Laplacian, that is
\begin{align*}
(-\Delta_r)^{s}u(x)\,&= C_{N,s,p}\lim_{\varepsilon\to 0}\displaystyle\int_{\mathbb{R}^N\setminus B_{\varepsilon}\left(x\right)}\dfrac{|u(x)-u(y)|^{r-2}(u(x)-u(y))}{|x-y|^{N+sr}}dy\\& =\, C_{N,s,p}\ {\bf P.V} \int_{\mathbb{R}^N} \dfrac{|u(x)-u(y)|^{p-2}(u(x)-u(y))}{|x-y|^{N+sr}}dy,\,\,\, \left(r = p,q; s = \alpha,\, \beta\right),
\end{align*}
and ${\bf P.V}$ denotes the commonly used abbreviations for
 ``in the main value sense'';  $\O \subset \R^N\ \left(N \geq 1\right)$ is a bounded domain in $\R^N$;  the weights $a(x), b(x)$ satisfy the following conditions:
\begin{align}\label{1.02}
\left\{\begin{array}{ll}
\noindent(i)\,\,\, a,\, b \in L^{\infty}\left(\O\right),\smallskip\\
(ii)\,\,\, \text{A\, Lebesgue\, measure\, sets\, of}\, \left\{x\in \O;\,\, a(x)>0\right\} \,,\, \left\{x\in \O;\,\, b(x)>0\right\}\,\text{are\, positive}.
\end{array}\right.
\end{align}
Throughout the paper, we will assume, without loss of generality, that $0<\alpha<\beta<1<q \leq p<\infty$; 
Before describing the main results, we first give some motivations.

Equation \eqref{1.1} is a natural extension of a following  $p\&q$ elliptic problem 
 \begin{align}\label{couple_laplacian}
 \Delta_p u+ \Delta_q u = f(x, u);\,\,\text{in}\,\,\O,\,\,\, u=0\,\,\,\text{on}\,\,\,\partial \O
 \end{align}
and quasilinear evolution reaction-diffusion equation
\begin{align}\label{reaction}
u_t = \text{div}\left[D(u)\nabla u\right]+ c\left(x, u\right)\,\,\, x\in \R^N,\,\, t>0,
\end{align}
where $D(u) = \left(\left|\nabla u\right|^{p-2}+ \left|\nabla u\right|^{q-2}\right)$. For a general term $D(u)$,
problem \eqref{reaction} has a wide range of applications in physics and related sciences such as biophysics \cite{PCF}, plasma physics \cite{HW}, and chemical reaction design \cite{RA, CI}, the operator  $-\Delta_p-\Delta$ has also appeared in quantum field theory, non linear optics, fluid mechanics, plasma
physics \cite{BAFP}. In this framework, the function $u$ describes a concentration, and the first term on the right-hand side of \eqref{reaction} corresponds to a diffusion with a diffusion coefficient $D(u)$; the term $c(x, u)$ stands for the reaction, related to sources and energy-loss processes. In particular, in chemical and biological applications, the reaction term
$c(x, u)$ has a polynomial form with respect to the concentration $u$ with variable coefficients (see \cite{GMF, LL, YY}). 

 In the past few years, a lot of attention has been given to the study of $(p,q)-$Laplace equations \cite{FMM, MP, S, YY}. In particular, the regularity of solutions of the $p\&q-$Laplacian problem has been studied by He and Li \cite{CG}. Using
the theory of regularity developed, the
authors showed that the weak solutions are locally $C^{
1, \alpha}$. Furthermore, P.
Baroni, G. Mingione and M. Colombo, (see \cite{BCM} and \cite{BCM1}), proved $C^{1, \alpha}$ regularity  when
the solutions are local minimizers for a class of integral functionals assuming that $1 < p \leq q$. An eigenvalue problem for the system of $-\Delta_p-\Delta$ equations 
was studied by Bence, Micheletti, Visetti \cite{BMV}, also  Colasuonno  and  Squassina \cite{cs} has studied the eigenvalue problem  of double phase variational
integrals and proved the existence of the eigenfunctions in Musielak–Orlicz space. The equation \eqref{couple_laplacian} with $c\left(x, u\right) = -p(x)\left|u\right|^{p-2}u - q(x)\left|u\right|^{q-2}u + \lambda g(x)\left|u\right|^{
\gamma-2}u$ for $1 < p < \gamma < q$ and $\gamma < p^*$
, where $p^* =\dfrac{Np}{N-p}$ if $p < n$, and $p^* = +\infty$, if $p \geq n$, was studied by Cherfils
and Il’yasov in \cite{CI} in bounded domain $\O \subset \R^N$. In \cite{ BZ1, BZ}, Benouhiba and Belyacine considered the equation \eqref{couple_laplacian} on the whole space $\R^N$ with $f(x, u) = \lambda g(x)\left|u\right|^{
p-2}u$ for $1 < q < p < q^*$ , and established existence of principal eigenvalue and a continuous family of generalized eigenvalues $\lambda$ under the key assumption 
 \[0\leq g\in L^{\left(\frac{q^*}{s}\right)'}\left(\R^N\right)\bigcap L^{\infty}\left(\R^N\right),\]
where $s$ is such that $\dfrac{p-t}{p^*}+ \dfrac{p\left(1-t\right)}{s} =1$, for some $t\in \left(0, \sqrt{\dfrac{p-q}{p}}\right) $.

In \cite{MT},  D. Motreanu, M. Tanaka studied the existence and non-existence of positive solutions for eigenvalue problem \eqref{couple_laplacian} with $f(x,u) = \lambda\left(m_p(x)\left|u\right|^{
p-2}u + m_q(x)\left|u\right|^{
q-2}u\right)$ for $1<q<p$ by using  the variational methods. Later, M. Tanaka \cite{Tanaka2} completely described the generalized eigenvalues $\lambda$ for which the
problem \eqref{couple_laplacian} with $f(x,u) = \lambda m_p(x)\left|u\right|^{
p-2}u$, has at least one positive solution. Moreover,  he also proved the existence of a sign-changing solution and the uniqueness of a positive solution for \eqref{couple_laplacian} in \cite{Tanaka1}. In the interesting paper \cite{bt}, when $m_p \equiv a,\, m_q\equiv b$ are constants,  Bobkov and Tanaka  provided a complete description of 2-dimensional sets in the $\left(a, b\right)$ plane, which becomes a threshold curve $\CC$ between the existence and non-existence of positive solutions for \eqref{couple_laplacian}.

Recently, the  research of fractional $(p, q)-$ Laplacian problems has drawn a significant  attention in the community of partial differential equations \cite{AT, AM,  AAI, BM, GKS} since it is not only a natural extension of the $(p,q)$-Laplacian equations but also presents many new phenomena described by nonlinear integral structures. In particular,  Ambrosio and Isernia \cite{AT}  considered fractional $p\& q$ laplacian problem with critical Sobolev-Hardy exponents 
\begin{align}\label{1.03}
\left\{\begin{array}{ll}
\left(-\Delta_p\right)^su + \left(-\Delta_q\right)^su \,= \dfrac{\left|u\right|^{p_s^*\left(\alpha\right)-2u}}{\left|x\right|^{\alpha}} +\lambda f(x,u) &\mbox{in } \Omega \smallskip\\
u=0&\mbox{in } \mathbb{R}^N \setminus \Omega\,,
\end{array}\right.
\end{align}
where $p_s^*\left(\alpha\right)= \dfrac{p\left(N-\alpha\right)}{N-sp}$ is the fractional critical Sobolev-Hardy exponents. They studied the existence of innitely many solutions for the problem \eqref{1.03} by using suitable concentration-compactness lemma and the Mountain Pass theorem.  Bhakta and Mukherjee  \cite{BM} studied the following problem in a bounded domain
\begin{align}\label{1.04}
\left\{\begin{array}{ll}
\left(-\Delta_p\right)^{s_1}u + \left(-\Delta_q\right)^{s_2}u \,= \left|u\right|^{p_{s_1}^*-2} +\theta V(x)\left|u\right|^{p-2}u+ \lambda f(x,u) &\mbox{in } \Omega, \smallskip\\
u=0&\mbox{in } \mathbb{R}^N \setminus \Omega\,,
\end{array}\right.
\end{align}
where $0 < s_2 < s_1 < 1 < r < q < p < \frac{n}{s_1}$
, and $V$ and $f$ are some appropriate functions.  They proved that there exist weak solutions of the problem \eqref{1.04} for some range of $\lambda, \theta$. Also, for $V(x) \equiv 1, \lambda = 0$ and assuming certain other conditions on $n, q, r$, they proved the existence of $cat_{\O}\left(\O\right)$ nonnegative solutions by using Lusternik-Schnirelmann category theory. 

Goel et al. \cite{GKS} studied the following fractional $(p, q)-$Laplacian problem
\begin{align}\label{1.05}
\left\{\begin{array}{ll}
\left(-\Delta_p\right)^{s_1}u + \beta\left(-\Delta_q\right)^{s_2}u \,= \lambda a(x)\left|u\right|^{\delta -2} +b(x)\left|u\right|^{r-2}u \, &\mbox{in } \Omega ,\smallskip\\
u=0&\mbox{in } \mathbb{R}^N \setminus \Omega\,,
\end{array}\right.
\end{align}
where $\O\subset \R^N$ is a bounded domain, $1 < \delta \leq q \leq p < r \leq p_{s_1}^* = \dfrac{np}{n - ps_1}
, 0 < s_2 <
s_1 < 1, n > ps_1, \lambda, \beta > 0$, and $a$ and $b$ are sign changing functions. Using Nehari manifold
method authors proved existence of at least two non-negative and non-trivial solutions in
the subcritical case for all $\beta > 0$ and for some range of $\lambda$. For the critical case under some
restriction on $\delta$, they obtained multiplicity results in some range of $\beta$ and $\lambda$. Furthermore,
they proved weak solutions of \eqref{1.05} are in the space $L^{\infty}\left(\O\right)\bigcap C^{0, \alpha}_{loc}\left (\O\right)$, for some $\alpha\in\left(0, 1\right)$,
when $2 \leq q \leq p < r < p_{s_1}^*$. Very recently, Alves, Ambrosio and Isernia \cite{AAI} studied the following class of
problems
\begin{align}\label{1.07}
\left(-\Delta_p\right)^{s}u + \left(-\Delta_q\right)^{s}u + V\left(\varepsilon x\right) \left(\left|u\right|^{p-2}u+ \left|u\right|^{q-2}u\right) = f(u) &\,\,\mbox{in}\,\, \R^N, 
\end{align}
with $s \in\left(0, 1\right)$ and $1 < p < q < N/s$. Under suitable assumptions on the potential and the nonlinearity, but without requiring the Ambrosetti–Rabinowitz condition, the authors  \cite{AAI} proved the existence of a ground state solution to \eqref{1.07} that concentrates around a minimum point of the potential $V$. Moreover, a multiplicity result is
established by using the Lyustenick–Schnirelmann category theory and the boundedness of solutions to \eqref{1.07}. On the other hand, Caponi and Pucci \cite{CP} used variational methods and Mountain pass theorem
to deal with existence of stationary Kirchhoff fractional p-Laplacian equations
\begin{align}\label{Kirchhoff}
 M([u]_{s,p}^p)(-\varDelta )^s_pu-\gamma \frac{|u|^{p-2}u}{|x|^{ps}}=\lambda w(x)|u|^{q-2}u+K(x)|u|^{p^*_s-2}u\quad \hbox {in }{\mathbb {R}}^N, 
 \end{align}
where $\gamma$ and $\lambda$ are real parameters, the exponent $q$ is such that $\theta p<q<p^*_s$ with $p_s^*=Np/(N-ps)$, and under suitable assumptions on the positive weights $w$ and $K$, they proved that there exists $\lambda^*>0$ such that for any $\lambda\geq\lambda^*$  then \eqref{Kirchhoff}  admits a non-trivial mountain pass solution while $\left\|K\right\|_{\infty}>0$, Theorem 1.1 in \cite{CP}. Furthermore, while $M(0) = 0, ps<N<2ps$ and $M$ satisfies some conditions, the authors \cite{CP} also proved the existence of solutions to the equation
\begin{align*}
 M([u]_{s,p}^p)(-\varDelta )^s_pu=\lambda w(x)|u|^{q-2}u+K(x)|u|^{p^*_s-2}u\quad \hbox {in }{\mathbb {R}}^N, \end{align*} 
 for any $\lambda\geq \lambda^*$ with $\lambda^*>0$, Theorem 1.2 \cite{CP}.

Motivated by the above papers, in this work we aim to  study the existence of nontrivial solutions
for a fractional $p\&q$ Laplacian problem involving two indefinite weigh \eqref{1.1}. As far as we know, the works  on fractional $p\&q$ Laplacian problems are  limited besides \cite{AT, AM,  AAI, BM, GKS}. In particular, the existence and non-existence of positive solution depending on the parameter $\lambda$  for the eigenvalue problem \eqref{1.1} still remained open. Our main contribution in this article is that we are able to characterize the existence and non-existence of problem \eqref{1.1}, which depends crucially on eigenspaces $\lambda_{1, p}^{a, \O}$ and $\lambda_{1,q}^{b, \O}$.  Recall that, by Pezzo and Quaas \cite{PQ}, $\lambda_{1, p}^{a, \O}$ is the positive, simple,  isolated  eigenvalue of problem
 
\begin{align}\label{1.9}
\left\{\begin{array}{ll}
\left(-\Delta_p\right)^{\alpha}\varphi_p^{a, \O}  \,= \lambda_{1, p} a(x) \left|\varphi_p^{a, \O}\right|^{p-2}\varphi_p^{a, \O}  \, &\mbox{in } \Omega, \smallskip\\
\varphi_p^{a, \O}=0&\mbox{on } \mathbb{R}^N \setminus \Omega,
\end{array}\right.
\end{align}
which is variationally characterized by
\begin{align*}
\lambda_{1,p}^{a, \O}: = \inf\left\{\dfrac{\displaystyle\int_{\O}\dfrac{\left|u(x) - u(y)\right|^p}{\left|x-y\right|^{N+\alpha p}}dxdy}{\displaystyle\int_{\O}a(x)\left|u\right|^pdx}: u\in W^{\alpha, p}_0\left(\O\right),\,\, \displaystyle\int_{\O}a(x)\left|u\right|^pdx>0\right\}.\end{align*}
This result was first considered by D. Motreanu, M. Tanaka \cite{MT} for the local case with $\alpha = \beta =1$.
One of the key tools in our analyze is the Rayleigh quotient, Lemma \ref{infimumeigenvalue} and mountain pass theorem, Lemma \ref{lem_2.4}. We prove that $\lambda_1^*$, which is defined by Rayleigh quotient, will not be attained. Otherwise, by using the Lagrange multiplier rule, we see that problem \eqref{1.1} admits a positive solution with $\lambda = \lambda_1^{*}\left(\O\right)$, but this contradicts Lemma \ref{lem.3.3}. On the other hand, using the mountain pass theorem to $\left(-\Delta_p\right)^{\alpha} +\left(-\Delta_q\right)^{\beta}$ operators, we show that existence of a positive solution to problem \eqref{1.1} with $\lambda> \lambda_1^{*}\left(\O\right)$, Theorem \ref{theorem_1.1}.  Our next result concerns the existence of principal eigenvalue and continuous family of positive eigenvalues  of the following problem
\begin{align}\label{generality}
(-\Delta)_p^{\alpha}u + (-\Delta)_q^{\beta}u \,= \lambda a(x) \left|u\right|^{p-2}u \, \,\,\,\,\mbox{in } \R^N,
\end{align}
where $0<\alpha<\beta<1<q < p< q_{\beta}^*$. To this aim, we first define the reflexive Banach space
\begin{align*} 
\CW = W^{\alpha, p}\left(\R^N\right)\bigcap W^{\beta, q}\left(\R^N\right),
\end{align*}
endowed with the norm $\left\|u\right\|_{\CW} = \left\|u\right\|_{W^{\alpha, p}\left(\R^N\right)}+ \left\|u\right\|_{W^{\beta, q}\left(\R^N\right)}$, where $\|\cdot\|_{W^{s, r}\left(\R^N\right)}$ is defined in \eqref{seminorm}. On the other hand,
the key assumptions imposed on the weight function is
 \begin{align}\label{1.11}
 0\leq a\in L^{\left(\frac{q_{\beta}^*}{s}\right)'}\left(\R^N\right)\bigcap L^{\infty}\left(\R^N\right),
 \end{align}
where $q_{\beta}^* = \dfrac{Nq}{N- q\beta}$ and $p_{\alpha}^* = \dfrac{Np}{N- p\alpha}$ are the fractional critical Sobolev exponents;  $s$ satisfies $\dfrac{p-t}{p_{\alpha}^*}+ \dfrac{p\left(1-t\right)}{s} =1$, for some $t\in \left(0, \sqrt{\dfrac{p-q}{p}}\right) $. This
condition plays a central role to prove the existence of a principal eigenvalue as well as a continuous family of positive eigenvalues.  These results are a considerable
extension of \cite{ BZ1, BZ}, which considered the case of $\alpha = \beta =1$.  We pointed out that the problem (\ref{generality}) is more interesting to consider in the whole space due to the lack of spatial compactness and the embedding  $W^{\alpha, p}\left(\R^N\right) \hookrightarrow W^{\beta, q}\left(\R^N\right)$ although our analysis can be applied to handle the problem in the bounded domain.

Our approach to obtain  the main results is replied on Mountain pass theorem, Lemma \ref{lem_2.4}. Another interesting way to prove the existence of solutions is Leray-Schauder alternative principle together with suitable truncation and
comparison techniques (see Papageorgiou, R$\breve{o}$dulescu and Repov$\check{s}$ \cite{PRR}). It is noted that the presence of two fractional $p,q$- Laplacian operators and the lack of compactness due to the critical exponent are one of the main difficulties in study the existence of positive solution and establishing the continuous spectrum.
More precisely, it is  not a direct application of Mountain-Pass Theorem to achieve the weak solutions
of \eqref{1.1} (resp. \eqref{generality}) due to lack of compactness for the Sobolev embedding. Thus, the boundedness of the Palais-Smale sequence is not easy to obtain. Furthermore,  the convergence of $\, \dfrac{\left|u_j(x)-u_j(y)\right|^{p-2}\left(u_j(x)-u_j(y)\right)}{\left|x-y\right|^{\frac{N+\alpha p}{p'}}}$ in $L^{p^{\prime}}\left(\R^{2N}\right)$ within $p^{\prime} = \dfrac{p}{p-1}$ does not ensure for existence of a Palais-Smale sequence. Thus, to overcome these difficulties, we use
the concentration-compactness principle due to P.L Lions \cite{PLL}  for tight sequences in $W^{\alpha,p}_0\left(\O\right)$, which was considered by C. O. Alves et.al. \cite{AAI}  and we show that the weak limits of Palais-Smale sequences of $J$ (resp. $I$) are also the weak solutions to \eqref{1.1} (resp. \eqref{generality}). 

Our main results read as follows :
\begin{theorem}\label{theorem_1.1}
Assume \eqref{1.02}  and $a,b\not\equiv 0$. Then the following conclusions hold.

i) If there exists $\lambda_1^*\left(\O\right)>0$ such that $0 < \lambda < \lambda_1^*\left(\O\right)$, then the problem \eqref{1.1} has no non-trivial solutions.

ii) If one of the following conditions hold:
\begin{enumerate}
\item $\lambda_{1, p}^{a, \O} \neq \lambda_{1, q}^{b, \O}$,
\item  There does not exist $k>0$ such that    $\varphi_p^{a, \O} = k\varphi_q^{b, \O}$ almost everywhere, then the problem \eqref{1.1}, with $\lambda = \lambda_1^*\left(\O\right)$ has no non-trivial solutions.
\end{enumerate}

iii) If $\lambda > \lambda_1^*\left(\O\right)$, and $\lambda_{1, p}^{a, \O} \neq \lambda_{1, q}^{b, \O}$ then the problem \eqref{1.1} has at least one positive solution.
\end{theorem}


Under the key assumption \eqref{1.11}, we further prove the existence of the eigenvalue for the problem \eqref{generality} and there exists a continuous family of   eigenvalues. More precisely, we obtain
\begin{theorem}\label{theorem_1.3}
Assume  $1 < q < p < q^*_{\beta}$ and \eqref{1.11}. Then problem \eqref{generality} admits a principal eigenvalue $\lambda_1\left(\alpha, \beta, p, q\right)\in \R$ which is simple and the associated with a positive eigenfunction. Moreover, for any $\lambda> \lambda_1\left(\alpha, \beta, p, q\right)$ is an eigenvalue of problem \eqref{generality}, i.e there exists a continuous set of positive eigenvalues.
\end{theorem}
We remark that Theorem \ref{theorem_1.3} confirms that the spectrum eigenvalue problem \eqref{generality} is continuous. This is in contrast with the result of Pezzo and Quass \cite{PQ}, who showed that the spectrum of the fractional $p-$Laplacian in $\R^N$, 
\begin{align}\label{PS1}
\left(-\Delta_p\right)^s u(x) = \lambda g(x)\left|u(x)\right|^{p-2}u(x)\,\,\,\,\text{in}\,\,\R^N,
\end{align}
 is discrete, i.e. there exists unbounded sequence of eigenvalues. Indeed, from Theorem 1.1  \cite{PQ},  under hypothesis $g\in L^{\infty}\left(\R^N\right)\bigcap L^{N/sp}\left(\R^N\right),\, N/sp>1$, the authors proved that  there exists a sequence of eigenpairs $\left\{\left(u_n, \lambda_n\left(g\right)\right)\right\}_{n\in\N}$ satisfying $\int\limits_{\R^N}g(x)\left|u_n(x)\right|^pdx=1$ and $0<\lambda_1(g)<\lambda_2(g)\leq\cdots\leq\lambda_n(g)\to+\infty$ if $g_+= \max\left\{0, g(x)\right\}\neq 0$. On the other hand, if $g_{\pm}\neq 0$ with $g_-=-\max\left\{0, -g(x)\right\}$ then  there exist  two sequences of eigenpairs $\left\{\left(u_n, \lambda^+_n\left(g\right)\right)\right\}_{n\in\N}$, $\left\{\left(v_n, \lambda^-_n\left(g\right)\right)\right\}_{n\in\N}$ such that $\int\limits_{\R^N}g(x)\left|u_n(x)\right|^pdx=1$, $\int\limits_{\R^N}g(x)\left|v_n(x)\right|^pdx=-1$ and 
 \begin{align*}
 &0<\lambda^{+}_1(g)<\lambda^{+}_2(g)\leq\cdots\leq\lambda^{+}_n(g)\to+\infty,\\& 0>\lambda^{-}_1(g)>\lambda^{-}_2(g)\geq\cdots\geq\lambda^{-}_n(g)\to-\infty.
 \end{align*}
Furthermore,  Theorem 1.3 \cite{PQ},  further stated that there is no positive principal eigenvalue for (\ref{PS1}) if $N/ps<1,\, g\in L^{\infty}\left(\R^N\right)$ and $ \int\limits_{\R^N}g(x)dx >0$.

\textit{The rest of this paper is organized as follows : In Section \ref{Preliminaries}, we recall some facts about the fractional Sobolev space that are used in the paper. Section \ref{section_3} is devoted to the study existence and non-existence solution
of the problem \eqref{1.1}. In Section \ref{section_4}, we establish the existence of principal eigenvalue for the problem \eqref{generality} and continuous family of eigenvalue result.}

\section{\bf Functional Setting and background}\label{Preliminaries}
Let $\O$  be a smooth (at least Lipschitz) domain of $\R^N$, and consider the Sobolev space of fractional order $s\in \left(0, 1\right)$ and exponent $r > 1$,
\begin{align*}
 W^{s,r}(\O):= \left\{u \in L^r(\O):\, \displaystyle\int_{\O}\displaystyle\int_{\O} \dfrac{|u(x)-u(y)|^r}{|x-y|^{N+sr}}dxdy \in L^{r}\left(\O\times\O\right)\right\},
 \end{align*}
which is endowed by the norm
\[\left\|u\right\|^r_{W^{s,r}\left(\O\right)} : = \left\|u\right\|_{L^r\left(\O\right)}^r + \displaystyle\int_{\O}\displaystyle\int_{\O} \dfrac{|u(x)-u(y)|^{r}}{|x-y|^{N+sr}}dxdy, \]
is a separable Banach space. Moreover, if $r \in (1,\infty )$ then $W^{s,r}\left(\O\right)$ is reflexive.

Let us consider  $W^{s, r}_0\left(\O\right)$ that is defined as follows:
\begin{align*}
W^{s, r}_0\left(\O\right): = \left\{ u \in L^r\left(\R^N\right): u = 0\,\,\text{in}\,\, \R^N\setminus \O\,\,\text{and}\,\, \left[u\right]_{s, r}< \infty\right\},\,\,\, 0<s<1<r,
\end{align*}
where
\begin{align}\label{norm}
\left[u\right]_{s,r}: = \left(\displaystyle\iint\limits_{\R^{2N}} \dfrac{|u(x)-u(y)|^{r}}{|x-y|^{N+sr}}dxdy\right)^{1/r}.
\end{align}
In fact, the Gagliardo semi-norm $\left[\cdot\right]_{s,r}$ is a norm in $W^{s, r}_0\left(\O\right)$ and that this
Banach space is uniformly convex. Actually,
\begin{align*}
W^{s, r}_0\left(\O\right)= \overline{C_c^{\infty}\left(\O\right)}^{\left[\cdot\right]_{s, r}}.
\end{align*}
When $\O = \R^N$, we simply write $\left|u\right|_r$. We deﬁne $\CD^{s,r}\left(\R^N\right)$ as the closure of $C_c^{\infty}\left(\R^N\right)$ with respect to
\begin{align}
\left[u\right]_{s,r}^r: = \displaystyle\iint\limits_{\R^{2N}} \dfrac{|u(x)-u(y)|^{r}}{|x-y|^{N+sr}}dxdy.
\end{align}
Let us deﬁne $W^{s,r}\left(\R^N\right)$ as the set of functions $u\in L^p\left(\R^N\right)$ such that $\left[u\right]_{s,r} < \infty$, endowed with the norm
\begin{align}\label{seminorm}
\left\|u\right\|^r_{W^{s,r}\left(\R^N\right)} = \left[u\right]_{s, r}^r + \left|u\right|_r^r.
\end{align}
\begin{remark}\label{remark_1}
Let $s \in (0, 1)$ and $N > sp$. Then there exists a constant $S_* > 0$ such that for
any $u \in \CD^{s,p}\left(\R^N\right)$
\begin{align}
\left|u\right|_{p_s^*}^p\leq S_*^{-1}\left[u\right]_{s, p}^p.
\end{align}
Moreover, $W^{s,p}\left(\R^N\right)$ is continuously embedded in $L^{t}\left(\R^N\right)$ for any $t \in \left[p, p_s^*\right]$ and compactly in $L^{t}\left(\CB_R(0)\right)$, for all $R > 0$ and for any $t\in \left[1, p_s^*\right)$.
\end{remark}

On the other hand, it is well-known that $W^{s,r}_{0}(\Gw)\hookrightarrow\hookrightarrow L^{\gamma}$ for any $\gamma \in \left(1, r^*_s\right)$. On the other hand $W^{s,r}_{0}(\Gw)\hookrightarrow L^{r^*_s}$,  if $rs \neq N$
and $W^{s,r}_{0}(\Gw)\hookrightarrow L^{\gamma}$ for all $\gamma \in (1, \infty)$ if $rs = N$ (see for example \cite{EG}). Moreover 
\begin{align}\label{2.02}
W^{s,r}_{0}(\Gw)\hookrightarrow\hookrightarrow C_0\left(\overline{\O}\right)\,\,\text{if}\,\,\, r>N/s.
\end{align}

(The notation $A\hookrightarrow\hookrightarrow B$ means that the continuous embedding $A\hookrightarrow B$ is compact.)
The compactness in \eqref{2.02} is consequence of the following Morrey’s type inequality (see \cite{EG})
\begin{align}\label{2.03}
\sup\limits_{\left(x,y\right)\neq\left(0,0\right)}\dfrac{\left|u(x)-u(y)\right|}{\left|x-y\right|^{s-\frac{N}{r}}}\leq C\left[u\right]_{s,r},\,\,\,\forall u\in W^{s,r}_0\left(\O\right),
\end{align}
which holds whenever $r > N/s$. If $r$ is sufficiently large, the positive constant $C$ in \eqref{2.03}  can be chosen uniform with respect to $r$.
We shall also include the p-Laplacian case and the solution space in this case is the usual
Sobolev space $W^{s,p}_{0}(\Gw)$ endowed with the norm $\|u\| = \left(\int_{\Gw}\left|\nabla u\right|^p\right)^{\frac{1}{q}}$. The notation
$W^{s,p}_{0}(\Gw),\,\,(s \in (0, 1])$ will be used to denote both the fractional Sobolev space defined above (when $0 < s < 1)$ and the usual Sobolev space $W^{s,p}_{0}(\Gw)$) (when $s = 1$). Note
that $p^*$ coincides with $p_1^*$.

Let $\left(-\Delta_r\right)^s$ be the $s-$fractional $r-$Laplacian, the operator acting from $W^{s,r}_{0}(\Gw)$  into its topological
dual, defined by
\begin{align*}
\left \langle \left(-\Delta_r\right)^su, \varphi \right \rangle_{s,r}= \displaystyle\int_{\R^N}\displaystyle\int_{\R^{N}}\left|u(x)- u(y)\right|^{r-2}\dfrac{\left(u(x)-u(y)\right)\left(\varphi(x)-\varphi(y)\right)}{\left|x-y\right|^{N+sr}}dxdy,\,\,\, \forall u,v \in W^{s,r}_{0}(\Gw).
\end{align*}
We recall that $\left(-\Delta_r\right)^su$ is the Gâteaux derivative at a function $u\in  W^{s,r}_{0}(\Gw)$  of the Fréchet differentiable functional $v \mapsto r^{-1}\left[v\right]_{s,r}^r$.
\begin{remark}
We take into account \eqref{2.02} and the following known facts:

$\left(j_1\right)$\,\,$W^{s,p}_{0}(\Gw)\not\hookrightarrow W^{s,q}_{0}(\Gw)$ for any $0 < s < 1 \leq q < p \leq \infty$ (see \cite{PW}, Theorem 1.1),

$\left(j_2\right)$\,\,$W^{\alpha,p}_{0}(\Gw)\hookrightarrow W^{\beta,q}_{0}(\Gw)$, whenever $0 < \alpha < \beta < 1 \leq p < q < \infty$ (see \cite{BL}, Lemma 2.6). 
\end{remark}
Thus, we assume that $\alpha, \beta, p$ and $q$ satisfy one of the following conditions:
\begin{align}\label{2.04}
0 < \alpha<\beta < 1 \,\,\,\text{and}\,\,\ \dfrac{N}{\alpha} < p < q,
\end{align}
or
\begin{align}\label{2.05}
0 <\beta< \alpha < 1 \,\,\,\text{and}\,\,\ \dfrac{N}{\beta} < q < p.
\end{align}
The assumption \eqref{2.04} provides the chain of embeddings
$W^{\beta,q}_{0}(\Gw) \hookrightarrow W^{\alpha,p}_{0}(\Gw)\hookrightarrow\hookrightarrow C_0\left(\overline{\O}\right)$ whereas \eqref{2.05} yields $ W^{\alpha,p}_{0}(\Gw)\hookrightarrow W^{\beta,q}_{0}(\Gw) \hookrightarrow\hookrightarrow C_0\left(\overline{\O}\right)$. 

 For $p > 1$ and $\alpha\in \left(0,1\right)$, we denote $\CO= \R^{2N}\setminus\left(C\O\right)^2$, where $C\O = \R^N\setminus\O$ and 
\begin{align*}
\CW^{\alpha, p} = \left\{u:\R^N\to \R \,\,\text{is\, measurable}\,\, u\in L^p\left(\O\right),\, \left(\displaystyle\int_{\CO} \dfrac{|u(x)-u(y)|^{p}}{|x-y|^{N+\alpha p}}dxdy\right)^{1/p} \right\},
\end{align*}
endowed with the norm as
\begin{align*}
\left\|u\right\|_{\CW^{\alpha, p}} = \left\|u\right\|_{L^p\left(\O\right)}+ \left(\displaystyle\int_{\CO} \dfrac{|u(x)-u(y)|^{p}}{|x-y|^{N+\alpha p}}dxdy\right)^{1/p}.
\end{align*}
Let $\CW^{\alpha, p}_0\left(\O\right)$  denote the closure of $C_c^{\infty}\left(\O\right)$ in $\CW^{\alpha, p}$, namely
\begin{align*}
\CW^{\alpha, p}_0\left(\O\right):= \left\{u \in \CW^{\alpha, p}: u = 0.\, \text{a.e. \, in}\,\, \R^N\setminus \O \right\} ,
\end{align*}
which is a uniformly convex Banach space endowed with the norm  (equivalent to $\left\|\cdot\right\|_{\CW^{\alpha, p}_0\left(\O\right)}$ ), namely
\begin{align*}
\left\|u\right\|_{\CW^{\alpha, p}_0\left(\O\right)} =\left[u\right]_{s,p, \CO}= \left(\displaystyle\int_{\CO} \dfrac{|u(x)-u(y)|^{p}}{|x-y|^{N+\alpha p}}dxdy\right)^{1/p}.
\end{align*}
Since $u= 0$ in $\R^N\setminus \O$, the above integral can be extended to all of $\R^N$. The embedding $\CW^{\alpha, p}_0\hookrightarrow L^r$ is continuous for any $r\in \left[1, p_{\alpha}^*\right]$ and compact for $r\in \left[1, p_{\alpha}^*\right)$. Further, for $0<\beta<\alpha<1< q \leq p$, then $\CW^{\alpha, p}_0\left(\O\right) \subset \CW^{\beta, q}_0\left(\O\right)$ (see \cite{BM}, Lemma 2.2). More details on $\CW^{\alpha, p}_0 \left(\O\right)$ and its properties can be found in \cite{EG}.
\begin{definition}\label{definition_1.1}
We say that $u \in \CW^{\alpha, p}_0\left(\O\right)$ is a (weak) solution of \eqref{1.1} if it holds
\begin{align*}
 &\displaystyle\int_{\R^N}\displaystyle\int_{\R^{N}}\left|u(x)- u(y)\right|^{p-2}\dfrac{\left(u(x)-u(y)\right)\left(v(x)-v(y)\right)}{\left|x-y\right|^{N+\alpha p}}dxdy\\& \hspace{1cm}+ \displaystyle\int_{\R^N}\displaystyle\int_{\R^{N}}\left|u(x)- u(y)\right|^{q-2}\dfrac{\left(u(x)-u(y)\right)\left(v(x)-v(y)\right)}{\left|x-y\right|^{N+\beta q}}dxdy= \lambda\displaystyle\int_{\O}\left(a(x) \left|u\right|^{p-2}u + b(x) \left|u\right|^{q-2}u\right)vdx
\end{align*}
for all $v \in \CW^{\alpha, p}_0\left(\O\right)$.
\end{definition} 

\begin{lemma}[Brezis-Lieb lemma]\label{lem.2.1}
If $\left\{u_n\right\}_{n\in\N}$ is a sequence weakly convergent to some $u$ in $\CW^{\alpha,p}_{0}(\Gw)$, then
\begin{align*}
\left[u_n-u\right]^p_{s,p,\CO} =\left[u_n\right]^p_{s,p, \CO} - \left[u\right]^p_{s,p, \CO}+o_n(1). 
\end{align*}
\end{lemma}
The proof of Lemma \ref{lem.2.1} can be found in \cite{AT}, Lemma 2.1.

\begin{lemma}[\cite{BM}, Lemma 2.2]
Let $0 < \alpha < \beta < 1 < q \leq p$ and $\O$ be a smooth bounded domain in $\R^N,$ 
where $\alpha p<N$. Then $\CW^{\alpha, p}_0\left(\O\right) \subset \CW^{\beta, q}_0\left(\O\right)$ and there exists $C = C\left(|\O|, N, p, q, \alpha, \beta\right) >0$ such that 
\begin{align*}
\left\|u\right\|_{\CW^{\beta, q}_0\left(\O\right)}\leq C\left\|u\right\|_{\CW^{\alpha, p}_0\left(\O\right)}.
\end{align*}
\end{lemma}
In the following sections we will use the mountain pass theorem to show the existence of a solution to
\eqref{1.1}. For the convenience of the reader, we recall the mountain pass theorem.
\begin{lemma}[Mountain Pass Theorem, see \cite{AR, SV, MW} ]\label{lem_2.4}
Let $X$ be a real Banach space and $\Phi \in C^1\left(X, \R\right)$. Suppose that $\Phi(0) = 0$ and
that there exist $\beta, \rho > 0$ and $x_1 \in X\setminus B_{\rho}(0)$ such that
\begin{enumerate}
\item $\Phi(u) \geq \beta$ for all $u\in X$ with $\left\|u\right\|_X = \rho$;
\item $\Phi\left(x_1\right)<\beta$.
\end{enumerate}
There exists a sequence $\left(u_n\right) \subset X$ satisfying
\begin{align*}
\Phi\left(u_n\right) \to c\,\,\,\text{and}\,\,\, \Phi^{\prime}\left(u_n\right)\to 0,
\end{align*}
where $c$ is the minimax level, defined by
\begin{align*}
c := \inf\left\{\max\limits_{t\geq 0}
\Phi\left(\gamma(t)\right) : \gamma \in C\left([0, 1], \R\right),\, \gamma(0) = 0\,\,\, \text{and}\,\,\, \gamma(1) = x_1\right\}.
\end{align*}
\end{lemma}

\section{\bf  Existence and non-existence of a positive solution}\label{section_3}
In this section, we present existence and non-existence results on a positive solution depend on $\lambda$  for the eigenvalue problem \eqref{1.1}.
 First, we define a functions $\Phi$ and $\Psi$ on $\CW^{\alpha,p}_0\left(\O\right)$ by
\begin{align}
\Phi(u) :&=
\dfrac{1}{p}\displaystyle\int_{\CO}\dfrac{\left|u(x) - u(y)\right|^p}{\left|x-y\right|^{N+\alpha p}}dxdy +\dfrac{1}{q}\displaystyle\int_{\CO}\dfrac{\left|u(x) - u(y)\right|^q}{\left|x-y\right|^{N+\beta q}}dxdy\nonumber\\& = \dfrac{1}{p}\left\|u\right\|_{\CW^{\alpha, p}_0\left(\O\right)} +\dfrac{1}{q}\left\|u\right\|_{\CW^{\beta, q}_0\left(\O\right)},
\end{align}
\begin{align}
\Psi(u):=\dfrac{1}{p}\displaystyle\int_{\O}a(x)\left|u\right|^pdx + \dfrac{1}{q}\displaystyle\int_{\O}b(x)\left|u\right|^qdx,\hspace{2,5cm}
\end{align}
for all $u\in \CW^{s,p}_0\left(\O\right)$.

Second, we define the Rayleigh quotient 
\begin{align}\label{1.7}
\lambda_1^*\left(\O\right) = \inf\left\{\dfrac{\Phi(u)}{\Psi(u)}; u\in \CW_0^{\alpha,p},\,\Psi(u)>0\right\}.
\end{align}


To prove the theorem \ref{theorem_1.1} , we first prove some lemmas.

\begin{lemma}\label{infimumeigenvalue}
Let 
\begin{align}\label{3.6}
\lambda_1^{*}\left(\O\right): = \inf\left\{\dfrac{\Phi(u)}{\Psi(u)}; u\in \CW_0^{\alpha, p}\left(\O\right), \Psi(u)>0 \right\}.
\end{align}
Then it holds $\lambda_1^{*}\left(\O\right) = \min\left\{\lambda_{1, p}^{a, \O}, \lambda_{1, q}^{b, \O}\right\}$.
\end{lemma}
\begin{proof}[\bf Proof]
Let $\varphi_p^{a, \O}$ and $\varphi_q^{b, \O}$ be a positive eigenfunctions corresponding to $\lambda_{1, p}^{a, \O}$ and $\lambda_{1, q}^{b, \O}$, respectively, such that $\displaystyle\int_{\O}a(x){\varphi_p^{a, \O}}^pdx = 1$ and $\displaystyle\int_{\O}b(x){\varphi_q^{b, \O}}^qdx = 1$. We take $t>0$ large enough such that
\begin{align*}
\Psi\left(t\varphi_p^{a, \O}\right)= t^q\left(t^{p-q} + \displaystyle\int_{\O}b(x){\varphi_p^{a, \O}}^q dx\right).
\end{align*}
Since $q<p$, we deduce
\begin{align*}
\lambda_1^{*}\left(\O\right)\leq \dfrac{\Phi\left(t\varphi_p^{a, \O}\right)}{\Psi\left(t\varphi_p^{a, \O}\right)} = \dfrac{\dfrac{t^p}{p}\left\|\varphi_p^{a, \O}\right\|^p_{\CW^{\alpha, p}_0\left(\O\right)} + \dfrac{t^q}{q}\left\|\varphi_p^{a, \O}\right\|^q_{\CW^{\beta,q}_0\left(\O\right)}}{t^q\left(t^{p-q} + \displaystyle\int_{\O}b(x){\varphi_p^{a, \O}}^q dx\right)} = \dfrac{\lambda_{1, p}^{a, \O} + \dfrac{p}{q}t^{q-p}\left\|\varphi_p^{a, \O}\right\|^q_{\CW^{\beta,q}_0\left(\O\right)}}{1+\dfrac{p}{q}t^{p-q}\displaystyle\int_{\O}{\varphi_p^{a, \O}}^qdx}\to \lambda_{1, p}^{a, \O},
\end{align*}
as $t\to +\infty$.

Similarly, we have
\begin{align*}
\lambda_1^{*}\left(\O\right)\leq \dfrac{\Phi\left(t\varphi_q^{b, \O}\right)}{\Psi\left(t\varphi_q^{b, \O}\right)} = \dfrac{\dfrac{t^p}{p}\left\|\varphi_q^{b, \O}\right\|^p_{\CW^{\alpha,p}_0\left(\O\right)} + \dfrac{t^q}{q}\left\|\varphi_q^{b, \O}\right\|^q_{\CW^{\beta,q}_0\left(\O\right)}}{t^q\left(t^{p-q} + \displaystyle\int_{\O}b(x){\varphi_q^{b, \O}}^q dx\right)} = \dfrac{\lambda_{1, q}^{b, \O} + \dfrac{q}{p}t^{p-q}\left\|\varphi_q^{b, \O}\right\|^q_{\CW^{\beta,q}_0\left(\O\right)}}{1+\dfrac{q}{p}t^{p-q}\displaystyle\int_{\O}{\varphi_q^{b, \O}}^qdx}\to \lambda_{1, q}^{b, \O},
\end{align*}
as $t\to 0^+$. We observe that $\lambda_1^{*}\left(\O\right) \leq\min \left\{\lambda_{1, p}^{a, \O} ,\lambda_{1, q}^{b, \O}\right \}$.
Subsequently, we claim that $\lambda_1^{*}\left(\O\right)\geq \min\left\{\lambda_{1, p}^{a, \O} ,\lambda_{1, q}^{b, \O}\right\}$.

 Assume by the contradiction that $\lambda_1^{*}\left(\O\right)< \min\left\{\lambda_{1, p}^{a, \O} ,\lambda_{1, q}^{b, \O}\right\}$. It follows by \eqref{3.6}, we have that there exists $u\in \CW^{\alpha, p}_0\left(\O\right)$ satisfies 
 \[\dfrac{\Phi(u)}{\Psi(u)} <\min\left\{\lambda_{1, p}^{a, \O} ,\lambda_{1, q}^{b, \O}\right\}, \,\, \Psi(u)>0.\]
 $\textit{Case (i)}$\,\, If $\displaystyle\int_{\O}a(x)\left|u\right|^pdx>0$ and $\displaystyle\int_{\O}b(x)\left|u\right|^qdx\leq 0$, there hold $p\Psi\left(u\right) \leq \displaystyle\int_{\O}a(x)\left|u\right|^pdx$ and 
 \begin{align*}
 p\Phi(u) =\left\|u\right\|^p_{\CW^{\alpha, p}_0\left(\O\right)} +\dfrac{p}{q}\left\|u\right\|^q_{\CW^{\beta, q}_0\left(\O\right)}> \left\|u\right\|^p_{\CW^{\alpha, p}_0\left(\O\right)}
 \end{align*}
 since $p>q$. From the definition of $\lambda_{1, p}^{a, \O}$, we get a contradiction, 
 \begin{align}\label{3.2}
\min\left\{\lambda_{1, p}^{a, \O} ,\lambda_{1, q}^{b, \O}\right\}>\dfrac{\Phi(u)}{\Psi(u)} \geq \lambda_{1, p}^{a, \O}. 
\end{align}
 $\textit{Case (ii)}\,\, $ If  $\displaystyle\int_{\O}a(x)\left|u\right|^pdx \leq 0$ and $\displaystyle\int_{\O}b(x)\left|u\right|^qdx\geq 0$.  As in the proof of the $\textit{Case} (i)$, we derive a contradiction from 
 \begin{align*}
 \min\left\{\lambda_{1, p}^{a, \O} ,\lambda_{1, q}^{b, \O}\right\}>\dfrac{\Phi(u)}{\Psi(u)} \geq \lambda_{1, q}^{b, \O}. 
 \end{align*}

 $\textit{Case (iii)}\,\,\ $ If  $\displaystyle\int_{\O}a(x)\left|u\right|^pdx > 0$ and $\displaystyle\int_{\O}b(x)\left|u\right|^qdx> 0$.  Similarly, we also have
 \begin{align}\label{3.3}
 \Phi(u)\geq \dfrac{\lambda_{1, p}^{a, \O}}{p}\displaystyle\int_{\O}a(x)\left|u\right|^pdx+ \dfrac{\lambda_{1, q}^{b, \O}}{q}\displaystyle\int_{\O}b(x)\left|u\right|^qdx\geq \min\left\{\lambda_{1, p}^{a, \O} ,\lambda_{1, q}^{b, \O}\right\}\Psi(u),
 \end{align}
 which leads to a contradiction.
 
  It is evident that
 \begin{align*}
 \min\left\{\lambda_{1, p}^{a, \O} ,\lambda_{1, q}^{b, \O}\right\}>\dfrac{\Phi(u)}{\Psi(u)} \geq \min\left\{\lambda_{1, p}^{a, \O}, \lambda_{1, q}^{b, \O}\right\}. 
 \end{align*}
According to all former cases, we deduce $\lambda_1^{*}\left(\O\right) = \min\left\{\lambda_{1, p}^{a, \O}, \lambda_{1, q}^{b, \O}\right\}$.
\end{proof}
\begin{lemma}\label{lemma_3.3}
Assume $\lambda_{1, p}^{a, \O}\neq \lambda_{1, q}^{b, \O}$ or in other words the corresponding eigenfunctions $\varphi_p^{a, \O}$ and $\varphi_q^{b, \O}$ are linearly independent. Then the infimum in \eqref{1.7} is not achieved.
\end{lemma}
\begin{proof}[\bf Proofs]
We need to show that the infimum in \eqref{1.7} is not achieved. Arguing indirectly we assume that this does not hold, and hence there exists $u\in \CW^{\alpha, p}_0\left(\O\right)$ such that $\Psi(u)>0$ and $\dfrac{\Phi(u)}{\Psi(u)}=\lambda_1^*$. It follows from Lemma \ref{infimumeigenvalue}, we obtain 
\begin{align}\label{eig}
\dfrac{\Phi(u)}{\Psi(u)}=\lambda_1^{*}\left(\O\right) = \min\left\{\lambda_{1, p}^{a, \O}, \lambda_{1, q}^{b, \O}\right\}.
\end{align}
We argue by considering the three cases in the proof of Lemma \ref{infimumeigenvalue}.

 $\textit{Case (i)}$\,\, Combining the equations \eqref{3.2}, \eqref{eig} , and incorporating hypotheses $\displaystyle\int_{\O}b(x)\left|u\right|^qdx\leq 0$, we have
 \begin{align*}
\lambda_1^{*}\left(\O\right) =\dfrac{\Phi(u)}{\Psi(u)}\geq \dfrac{\left\|u\right\|^p_{\CW^{\alpha, p}_0\left(\O\right)} +\dfrac{p}{q}\left\|u\right\|^q_{\CW^{\beta, q}_o\left(\O\right)}}{\displaystyle\int_{\O}b(x)\left|u\right|^qdx}\geq \dfrac{\left\|u\right\|^p_{\CW^{\alpha,p}_0\left(\O\right)} }{\displaystyle\int_{\O}b(x)\left|u\right|^qdx}\geq \lambda_{1, q}^{b, \O}\geq \lambda_1^{*}\left(\O\right).
 \end{align*}
We deduce 
\begin{align*}
\left\|u\right\|^p_{\CW^{\alpha,p}_0\left(\O\right)}   = \lambda_{1, q}^{b, \O}\displaystyle\int_{\O}b(x)\left|u\right|^qdx\,\,\text{and}\,\, \left\|u\right\|^p_{\CW^{\alpha,p}_0\left(\O\right)}  =0.
\end{align*}
Therefore, $u = 0$. This contradicts with assumption that  $u\neq 0$.

$\textit{Case (ii)}$\,\,\, As in the proof of $\textit{Case (i)}$, we also get a contradiction.

$\textit{Case (iii)}$ \,\, From the equations \eqref{3.3} and \eqref{eig}, we deduce 
\begin{align}
\dfrac{\left\|u\right\|^p_{\CW^{\alpha,p}_0\left(\O\right)} }{\displaystyle\int_{\O}a(x)\left|u\right|^pdx}= \lambda_{1, p}^{a, \O}= \lambda_{1, q}^{b, \O} = \dfrac{\left\|u\right\|^q_{\CW^{\beta, q}_0\left(\O\right)}}{\displaystyle\int_{\O}b(x)\left|u\right|^qdx}.
\end{align}
It follows that $u =t\varphi_p^{a, \O} = s\varphi_q^{b, \O}$ for some $t\neq 0$ and $s\neq 0$. This is a contradiction to our hypothesis.
\end{proof}
Next we consider the following energy functional
\begin{align}\label{3.03}
J(u) = \Phi(u) - \lambda\Psi(u)
\end{align}
to study weak solutions to \eqref{1.1}. Clearly,  $J$ belongs to $C^1\left(\CW^{\alpha, p}_0\left(\Gw\right),\R\right)$, and its differential is given by
\begin{align}
\left(J^{\prime}(u),\varphi\right) =& \displaystyle\int_{\CO}\left[\dfrac{\left|u(x) - u(y)\right|^{p-2}}{\left|x-y\right|^{N+\alpha p}} + \dfrac{\left|u(x) - u(y)\right|^{q-2}}{\left|x-y\right|^{N+\beta q}}\right]\left(u(x)-u(y)\right)\left(\varphi(x)- \varphi(y)\right)dxdy\nonumber&\\
&-\lambda\left[\displaystyle\int_{\O}a(x)\left|u\right|^{p-2} + b(x)\left|u\right|^{q-2}\right]u\varphi dx.
\end{align}

We organize the proof of this existence result in several lemmas. First, we prove that the functional $J$ satisfies the following local Palais-Smale condition. In the following lemma, Lemma \ref{lem.3.2}, we use the symbol $O_n(1)$, which is the quantity tending $0$ when $n\to \infty$. 
\begin{lemma}\label{lem.3.2}
Let $\alpha p<N$ and $0\leq \lambda \neq \lambda_{1, p}^{a, \O}$. If $\{u_n\}$ is a bounded $(PS)_c$
sequence of the functional $J$ defined by \eqref{3.03} , then the functional $J$ satisfies $(PS)_c$ condition.
\end{lemma}
\begin{proof}[\bf Proof]
Let $\{u_n\}$ be a $(PS)_c$ sequence, that is
\begin{align}
J\left(u_n\right)\to c \ \ \ \textrm{and}\ \ \left\|J^{\prime}\left(u_n\right)\right\|_{\left(\CW^{\alpha, p}_0\left(\O\right)\right)^{\prime}}\to 0\ \ \ \textrm{as}\ \ \ n\to\infty.
\end{align}
 By standard arguments we can
show that $\left\{u_n\right\}$ is bounded in $\CW^{\alpha,p}_{0}\left(\Gw\right)$, there exists a subsequence, still denoted by $\{u_n\}$,
such that
\begin{align}\label{3.8}
& u_n \rightharpoonup  u\ \ \textrm{weakly\ in}\ \ \CW^{\alpha,p}_0\left(\O\right), \,\, u_n \to u \,\, a.e \,\,\textrm{in}\,\, \R^N,&\\\nonumber
& u_n \to u \ \ \textrm{strongly\ in}\ \ L^{\gamma}\left(\O\right)\,\, \textrm{for}\,\, 1\leq \gamma< p^*_{\alpha}.
\end{align}
To begin with, we prove  $u_n \to u$ in $\CW^{\alpha,p}_0\left(\O\right) $.  Let us fix $\varphi \in \CW^{\alpha,p}_0\left(\O\right)$  and denote by 
\begin{align}
&A_p\left(u,\varphi\right) = \displaystyle\int_{\CO}\dfrac{\left|u(x) - u(y)\right|^{p-2}(u(x)- u(y))}{\left|x-y\right|^{N+\alpha p}}\varphi(x)- \varphi(y)dxdy, &\\ \nonumber
& B_q\left(u,\varphi\right) = \displaystyle\int_{\CO}\dfrac{\left|u(x) - u(y)\right|^{q-2}(u(x)- u(y))}{\left|x-y\right|^{N+\beta q}}\varphi(x)- \varphi(y)dxdy,
\end{align}
where $A_p\left(u,\varphi\right)$ and $B_q\left(u,\varphi\right)$ are the linear functions.  Clearly,  by H$\ddot{o}$lder inequality then $A_p\left(u,\varphi\right) + B_q\left(u,\varphi\right)$ is also continuous and
\begin{align}
\left|A_p\left(u,\varphi\right) + B_q\left(u,\varphi\right)\right|&\leq \left|A_p\left(u,\varphi\right) \right|+\left| B_q\left(u,\varphi\right)\right|\leq \left[v\right]^{p-1}_{\alpha, p, \CO}\left[\varphi\right]_{\alpha, p, \CO}+ \left[v\right]^{q-1}_{\beta, q, \CO}\left[\varphi\right]_{\beta, q, \CO}&\\ \nonumber
&\leq \left( \left\|v\right\|^{p-1}_{\CW^{\alpha, p}_0\left(\O\right)}+\left\|v\right\|^{q-1}_{\CW^{\beta, q}_0\left(\O\right)} \right)\left\|\varphi\right\|_{\CW^{\alpha, p}_0\left(\O\right)}.
\end{align}
Since $u_n \rightharpoonup  u$ weakly in $\CW^{\alpha, p}_0\left(\O\right)$, we deduce that
 $\lim_{n\to\infty} A_p\left(u,\varphi\right) =\lim_{n\to\infty} B_q\left(u,\varphi\right) = 0$, and 
 \begin{align}
 \lim_{n\to\infty}[A_p\left(u_n-u,u\right)+ B_q\left(u_n-u,u\right)] = 0.
 \end{align}
 It is noticed that
 \begin{align}\label{3.12}
 O_n(1)& = \left(J^{\prime}\left(u_n\right) - J(u), u_n -u\right)&\\ \nonumber
 & = \left[A_p\left(u_n-u, u_n\right)+ B_q\left(u_n-u, u_n\right)\right]- \left[A_p\left(u_n-u, u\right)+ B_q\left(u_n-u,u\right)\right]-\lambda \Lambda_n,
 \end{align}
  where
 \begin{align*} 
 \Lambda_n = \displaystyle\int_{\O}\left[a(x)\left(\left|u_n\right|^{p-2}u_n - \left|u\right|^{p-2}u\right) + b(x)\left(\left|u_n\right|^{q-2}u_n - \left|u\right|^{q-2}u\right)\right]\left(u_n-u\right) dx .
 \end{align*}
 as  $n\to \infty$. It follows by \eqref{3.8}, we have $\Lambda_n \to 0$ as $n\to \infty$.

Let us now recall the well-known vector inequalities.
 For $1<p \leq 2$, thanks to the inequality
 \begin{align}\label{3.13}
 \dfrac{\left|\xi - \eta\right|^p}{\left(\left|\xi\right|^2 + \left|\eta\right|^{\frac{2-p}{2}}\right)} \leq C_1(p)\left(\left|\xi\right|^{p-2}\xi -\left|\eta\right|^{p-2}\eta \right)^{\frac{p}{2}}\left(\xi -\eta\right)^{\frac{p}{2}},\,\,\,\forall \xi, \eta \in \R^N.
 \end{align}
 On the other hand, for $2<p<+\infty$, we have
\begin{align}\label{3.14}
\left|\xi - \eta \right|^p \leq  C_2(p) \left(\left|\xi\right|^{p-2}\xi -\left|\eta\right|^{p-2}\eta \right)\left(\xi -\eta\right),\,\,\forall \xi, \eta \in \R^N.
\end{align}
 It is noticed that $C_1(p)$ and $C_2(p)$ are positive constants depending only on $p$.
 
We assume that $p > q > 2$. We obtain  \begin{align}\label{3.15}
\nonumber  \left[u_n - u\right]^p_{\alpha,p, \CO} &= \displaystyle\int_{\CO}\dfrac{\left|u_n(x) - u_n(y)-(u(x)- u(y))\right|^{p}}{\left|x-y\right|^{N+\alpha p}}dxdy\\&\nonumber
 \leq C_2(p)\displaystyle\int_{\CO}\left[\left|u_n(x) - u_n(y)\right|^{p-2}\left(u_n(x) - u_n(y)\right) -\left|u(x) - u(y)\right|^{p-2}\left(u(x) - u(y)\right) \right]\\&\times \left(u_n(x) - u_n(y)-(u(x)- u(y))\right)\left|x-y\right|^{-N-\alpha p}dxdy\\&\nonumber
\leq C_2(p)\left[ A_p\left(u_n-u,u_n\right) - A_p\left(u_n-u,u\right)\right].
  \end{align}
Similarly,
\begin{align}\label{3.16}
\nonumber  \left[u_n - u\right]^q_{\beta,q, \CO} &= \displaystyle\int_{\CO}\dfrac{\left|u_n(x) - u_n(y)-(u(x)- u(y))\right|^{q}}{\left|x-y\right|^{N+\beta q}}dxdy\\&\nonumber
 \leq C_2(q)\displaystyle\int_{\O}\left[\left|u_n(x) - u_n(y)\right|^{q-2}\left(u_n(x) - u_n(y)\right) -\left|u(x) - u(y)\right|^{q-2}\left(u(x) - u(y)\right) \right]\\&\times \left(u_n(x) - u_n(y)-(u(x)- u(y))\right)\left|x-y\right|^{-N-\beta q}dxdy\\&\nonumber
\leq C_2(q)\left[ B_q\left(u_n-u,u_n\right) - B_q\left(u_n-u,u\right)\right].
  \end{align}
 Let $C_0= \min\left\{C_2^{-1}(p), C_2^{-1}(q)\right\}$. From the equations \eqref{3.15} and \eqref{3.16}, we have
 \begin{align}
  A_p\left(u_n-u, u_n\right)- A_p\left(u_n-u, u\right) + B_q\left(u_n-u, u_n\right)- B_q\left(u_n-u,u\right)\geq C_0\left(\left[u_n - u\right]^q_{\beta, q, \CO}+\left[u_n - u\right]^p_{\alpha, p, \CO} \right).
 \end{align}
It follows from \eqref{3.12} that
 \begin{align}
 C_0\left(\left[u_n-u\right]^p_{\alpha, p, \CO} + \left[u_n-u\right]^q_{\beta, q, \CO}\right)\leq O_n(1)+ \lambda\Lambda_n.
 \end{align}
Hence, we have $u_n \to u$ in $\CW^{\alpha,p}_0\left(\O\right)$ as $n\to \infty$.
 
In the case $1<p\leq 2$. By \eqref{3.8}, we deduce that there exists $M>0$ such that $\left[u_n\right]_{\alpha, p, \CO}\leq M$. Then from the equation \eqref{3.13} and the Hölder inequality, it follows that
\begin{align}\label{3.20}
 \left[u_n - u\right]^p_{\alpha, p, \CO} &\leq C_3(p)\left[ A_p\left(u_n-u,u_n\right) - A_p\left(u_n-u,u\right)\right]^{\frac{p}{2}}\left(\left[u_n\right]_{\alpha, p, \CO}^{\frac{p\left(2-p\right)}{2}}+ \left[u\right]_{\alpha, p, \CO}^{\frac{p\left(2-p\right)}{2}}\right)\nonumber\\& \leq
 C_4(p)\left[ A_p\left(u_n-u,u_n\right) - A_p\left(u_n-u,u\right)\right]^{\frac{p}{2}}.
\end{align}
Similarly, for $1 < q < 2$, we have
\begin{align}\label{3.21}
 \left[u_n - u\right]^q_{\beta, q, \CO} &\leq C_5(q)\left[ B_q\left(u_n-u,u_n\right) - B_q\left(u_n-u,u\right)\right]^{\frac{q}{2}}\left(\left[u_n\right]_{\beta, q, \CO}^{\frac{q\left(2-q\right)}{2}}+ \left[u\right]_{\beta, q, \CO}^{\frac{q\left(2-q\right)}{2}}\right)\nonumber\\& \leq
 C_6(q)\left[ B_q\left(u_n-u,u_n\right) - B_q\left(u_n-u,u\right)\right]^{\frac{q}{2}}.
\end{align}
Combining Eqs. \eqref{3.20}, \eqref{3.21}, we obtain
 \begin{align}
  A_p\left(u_n-u, u_n\right)- A_p\left(u_n-u, u\right) + B_q\left(u_n-u, u_n\right)- B_q\left(u_n-u,u\right)\geq C_1\left(\left[u_n - u\right]^2_{\alpha, p, \CO} + \left[u_n - u\right]^2_{\beta, q, \CO}\right),
 \end{align}
 with some $C_1>0$. This proves $u_n \to u$ strongly in $\CW^{\alpha, p}_0\left(\O\right)$ as $n\to \infty$.
 
  Therefore,
$J$ satisfies the $(PS)_c$ condition in $\CW^{\alpha, p}_0\left(\O\right)$.
\end{proof}

\begin{lemma}\label{lem.3.3}
Assume that $\lambda_{1, q}^{b, \O}< \lambda<\lambda_{1, p}^{a, \O}$, then the function $J$ is coercive.
\end{lemma}
\begin{proof}[\bf Proof]
Let $u\in \CW^{\alpha, p}_0\left(\O\right)$ with $\displaystyle\int_{\O}a(x)u_+^pdx \leq 0$, by using H$\ddot{o}$lder's inequality and fractional Sobolev embedding, we obtain
\begin{align}
J(u)& = \dfrac{1}{p}\displaystyle\int_{\CO}\dfrac{\left|u(x) - u(y)\right|^p}{\left|x-y\right|^{N+\alpha p}}dxdy +\dfrac{1}{q}\displaystyle\int_{\CO}\dfrac{\left|u(x) - u(y)\right|^q}{\left|x-y\right|^{N+\beta q}}dxdy - \dfrac{\lambda}{p}\displaystyle\int_{\O}a(x)\left|u\right|^pdx - \dfrac{\lambda}{q}\displaystyle\int_{\O}b(x)\left|u\right|^qdx \nonumber\\& 
\geq \dfrac{1}{p}\displaystyle\int_{\CO}\dfrac{\left|u(x) - u(y)\right|^p}{\left|x-y\right|^{N+\alpha p}}dxdy +\dfrac{1}{q}\displaystyle\int_{\CO}\dfrac{\left|u(x) - u(y)\right|^q}{\left|x-y\right|^{N+\beta q}}dxdy - \dfrac{\lambda}{q}\displaystyle\int_{\O}b(x)\left|u\right|^qdx \nonumber\\& 
\geq \dfrac{1}{p}\left\|u\right\|^p_{\CW^{\alpha,   p}_0\left( \O\right)}  - \dfrac{\lambda \left\|b\right\|_{\infty}\left|\O\right|^{1-q/p} }{q{\lambda_{1, p}^{a, \O}}^{q/p}}\left\|u\right\|^q_{\CW^{\alpha,   p}_0\left( \O\right)}\geq C_1 \left\|u\right\|^p_{\CW^{\alpha,   p}_0\left( \O\right)},
\end{align}
where $C_1 = \min\left\{\dfrac{1}{p},\dfrac{\lambda \left\|b\right\|_{\infty}\left|\O\right|^{1-q/p} }{q{\lambda_{1, p}^{a, \O}}^{q/p}} \right\}$.

Now, let $u\in \CW^{\alpha, p}_0\left(\O\right)$ with $\displaystyle\int_{\O}a(x)u_+^pdx > 0$. By our assumption of $\lambda <\lambda_{1, p}^{a, \O}$, we may fix $\varepsilon>0$ such that 
\begin{align*}
\left(1-\varepsilon\right)\lambda_{1, p}^{a, \O}>\lambda.
\end{align*}
Hence, in turn, we get
\begin{align}
J(u)& \geq \dfrac{\varepsilon}{p}\displaystyle\int_{\CO}\dfrac{\left|u(x) - u(y)\right|^p}{\left|x-y\right|^{N+\alpha p}}dxdy +\dfrac{\left(1-\varepsilon\right)\lambda_{1, p}^{a, \O}-\lambda}{q}\displaystyle\int_{\O}a(x)u_+^pdx - \dfrac{\lambda}{q}\displaystyle\int_{\O}b(x)\left|u\right|^qdx\nonumber\\& \geq \dfrac{\varepsilon}{p}\left\|u\right\|^p_{\CW^{\alpha, p}_0\left(\O\right)}  - \dfrac{\lambda \left\|b\right\|_{\infty}\left|\O\right|^{1-q/p} }{q{\lambda_{1, p}^{a, \O}}^{q/p}}\left\|u\right\|^q_{\CW^{\alpha, p}_0\left(\O\right)}\geq C_2 \left\|u\right\|^p_{\CW^{\alpha, p}_0\left(\O\right)},
\end{align}
where $C_2 = \min\left\{\dfrac{\varepsilon}{p},\dfrac{\lambda \left\|b\right\|_{\infty}\left|\O\right|^{1-q/p} }{q{\lambda_{1, p}^{a, \O}}^{q/p}} \right\}$.

Hence, $J$ is coercive and bounded from below.
\end{proof}
In the sequel, we show that if $0<\alpha<\beta<1 < q \leq p$ and $a, b$\, satisfy \eqref{1.02}, then $J$ possesses a mountain pass geometry.
\begin{lemma}\label{lem.3.4} Assume $0<\beta<\alpha\leq 1 < q < p$, then there exists $\lambda_*>0$ such that
\begin{enumerate}
\item There exists $\delta > 0$ and $\rho >0$ such that for all $\lambda \in \left(0, \lambda_*\right)$,
$J(u)\geq \delta$ on $\left\|u\right\|_{\CW^{\alpha, p}_0\left(\O\right)} = \rho$,\\
\item  There exists $u_0 \in \CW^{\alpha,p}_0\left(\O\right)$ with $\left\|u_0\right\|_{\CW^{\alpha, p}_0\left(\O\right)}>\rho$ and $J(u_0)< 0$.
\end{enumerate}
\end{lemma}
\begin{proof}[\bf Proof]
Since $a, b \in L^{\infty}\left(\O\right)$. By the Hölder inequality  and Sobolev embeddings, we have
\begin{align}
J(u)&\geq \dfrac{1}{p}\left\|u\right\|^p_{\CW^{\alpha, p}_0\left(\O\right)} + \dfrac{1}{q}\left\|u\right\|^q_{\CW^{\beta, q}_0\left(\O\right)} - \lambda \dfrac{C_1}{p}\left\|u\right\|_{\CW^{\alpha,p}_0\left(\O\right)}^p- \lambda \dfrac{C_2}{q}\left\|u\right\|_{\CW^{\beta, q}_0\left(\O\right)}^q\nonumber\\&
\geq \dfrac{1}{p}\left\|u\right\|^p_{\CW^{\alpha, p}_0\left(\O\right)}  - \lambda \dfrac{C_1}{p}\left\|u\right\|^p_{\CW^{\alpha, p}_0\left(\O\right)}- \lambda \dfrac{C_2}{q}\left\|u\right\|^q_{\CW^{\alpha, p}_0\left(\O\right)}\nonumber\\&
= \left\|u\right\|^q_{\CW^{\alpha, p}_0\left(\O\right)}\left(\dfrac{1}{p}\left\|u\right\|^{p-q}_{\CW^{\alpha, p}_0\left(\O\right)}- \lambda \dfrac{C_2}{q}- \lambda \dfrac{C_1}{p}\left\|u\right\|_{\CW^{\alpha, p}_0\left(\O\right)}^{p-q}\right).
\end{align}
We claim that $\dfrac{1}{p}\left\|u\right\|^{p-q}_{\CW^{\alpha, p}_0\left(\O\right)}- \lambda \dfrac{C_2}{q}- \lambda \dfrac{C_1}{p}\left\|u\right\|_{\CW^{\alpha, p}_0\left(\O\right)}^{p-q}>0$. 

Indeed, if we choose $\lambda_{*} = \dfrac{1}{p}\dfrac{t_0^{p-q}}{\dfrac{C_2}{q}+\dfrac{C_1}{p} t_0^{p-q} }$ with $t_0>0$, we denote  
\begin{align}
\delta = \left(\dfrac{1}{p} -\lambda_{*}\dfrac{C_1}{p}\right)t_0^{p-q} - \lambda_{*}\dfrac{C_2}{q} >0,
\end{align}
then for every $ \lambda\in \left(0, \lambda_{*}\right)$,  we have $\left(\dfrac{1}{p} -\lambda\dfrac{C_1}{p}\right)t_0^{p-q} - \lambda\dfrac{C_2}{q}>  \delta$.

We deduce  $J(u)\geq \delta$ for all $u\in \CW^{\alpha,p}_0\left(\O\right)$ when we choose $\rho = t_0$. Hence,
(1) is verified.

Since $J\left(tu\right) \to -\infty$ as $t \to \infty$, we therefore can find $\tau > 0$ sufficiently large such that $
\left\|\tau u\right\|_{\CW^{\alpha, p}_0\left( \O\right)}> \rho$ and $J\left(\tau u\right) < 0$. This
fact shows that (2) holds true.
\end{proof}
Now, we can prove Theorem \ref{theorem_1.1}.
\begin{proof}[\bf Proof Theorem \ref{theorem_1.1}]
i)
Assume by contradiction that there exists a non-trivial solution $u$ of problem \eqref{1.1}. Then, for every $s > 0$, we have that $v = su$ is also a non-trivial solution of problem \eqref{1.1} within coefficient slight changing. We can choose $s^{p-q} = \dfrac{p}{q}$ and then act
with $su$ as test function on the problem \eqref{1.1}. We have
\begin{align*}
0<p\Phi(su) = q\lambda\Psi(su).
\end{align*}
This is together with Lemma \ref{infimumeigenvalue}, we obtain
\begin{align*}
\lambda = \dfrac{\Phi(su)}{\Psi(su)} \geq \lambda_1^* = \min\left\{\lambda_{1, p}^{a, \O}, \lambda_{1, q}^{b, \O}\right\}.
\end{align*}
We get a contradiction and this confirms the first assertion of the theorem.

ii) The second conclusion $(ii)$ follows directly from Lemma \ref{lemma_3.3}.

iii) First, we assume that $\lambda_{1, q}^{b, \O} <\lambda $.  The function $J(u)$ is sequentially weakly lower semi-continuous. Indeed, if $u_n \rightharpoonup u$ in $\CW^{\alpha,p}_0\left(\O\right)$ as $n\to \infty$, so 
\begin{align*}
J(u)&\leq \dfrac{1}{p}\liminf_{n\to\infty}\left[u_n\right]_{\alpha, p, \CO}+ \dfrac{1}{q}\liminf_{n\to\infty}\left[u_n\right]_{\beta, q, \CO}-\liminf_{n\to\infty}\left[\dfrac{\lambda}{p}\displaystyle\int_{\O}a(x)\left|u_n\right|^pdx + \dfrac{\lambda}{q}\displaystyle\int_{\O}b(x)\left|u_n\right|^qdx\right]\\&\leq \liminf_{n\to\infty}\left(\dfrac{1}{p}\left[u_n\right]_{\alpha, p, \CO}+ \dfrac{1}{q}\left[u_n\right]_{\beta, q, \CO}-\dfrac{\lambda}{p}\displaystyle\int_{\O}a(x)\left|u_n\right|^pdx - \dfrac{\lambda}{q}\displaystyle\int_{\O}b(x)\left|u_n\right|^qdx\right)\\& = \liminf_{n\to\infty}J\left(u_n\right).
\end{align*}
Followed by Lemma \ref{lem.3.3} that $J\left(u\right)$ is coercive and bounded from below. Therefore, by a standard result (see, eg \cite{ATG}, Theorem 1.1), there exists a local minimizer $u_0$ of $J(u)$.

Second, in order to proof $u_0\neq 0$, we show that $J\left(u_0\right)= \min_{\CW^{\alpha,p}_0\left(\O\right)}J <0$. 

 Indeed, let $\varphi_p^{a, \O}$ be the eigenfunction corresponding to $\lambda_{1, q}^{b, \O}$ that satistifies $\displaystyle\int_{\O}b(x){\varphi_p^{a, \O}}^qdx = 1$. Because $\lambda>\lambda_{1, q}^{b, \O}$, for sufficiently small $\tau>0$ it holds
\begin{align*}
J\left(t\varphi_p^{a, \O}\right) = \tau^q\left(\dfrac{\tau^{p-q}}{p}\left[\varphi_p^{a, \O}\right]_{\alpha, p, \CO} - \dfrac{\lambda \tau^{p-q}}{p}\displaystyle\int_{\O}a(x){\varphi_p^{a, \O}}^pdx + \dfrac{\lambda_{1, q}^{b, \O} - \lambda}{q}\right)<0.
\end{align*} 
Later, suppose $\lambda_{1, p}^{a, \O}< \lambda$. Recalling that $J$ satisfies the Palais-Smale condition by virtue of Lemma \ref{lem.3.2}, and followed by Lemma \ref{lem.3.4} allow us to apply the mountain pass theorem, which guarantees the existence of a critical value $c\geq \delta$ of $J$, with $\delta >0$, namely
\begin{align*}
c: = \inf_{\gamma\in \Sigma}\max_{\tau\in\left[0,1\right]}J\left(\gamma(\tau)\right),
\end{align*}
\[\Sigma: = \left\{\gamma\in C\left([0,1],\,\, \CW^{\alpha,p}_0\left(\O\right)\right): \gamma(0) = 0, \gamma(1) = \tau\varphi_p^{a, \O}\right\}.\]
This completes the proof of Theorem \ref{theorem_1.1}.
\end{proof}
\section{\bf Principal eigenvalue}\label{section_4}
In this section, we show that existence the principal eigenvalue and continuous family of eigenvalue of problem \eqref{generality}.
\subsection{ Existence principal eigenvalue  }
We say that $\lambda$ is a principal eigenvalue of problem \eqref{generality} if there exists a eigenfunction $u \in\CW,\,\, u\neq 0$ and $u \geq 0$ that satisfies
\begin{align*}
 &\displaystyle\int_{\R^N}\displaystyle\int_{\R^{N}}\left|u(x)- u(y)\right|^{p-2}\dfrac{\left(u(x)-u(y)\right)\left(v(x)-v(y)\right)}{\left|x-y\right|^{N+\alpha p}}dxdy\\& \hspace{4cm}+ \displaystyle\int_{\R^N}\displaystyle\int_{\R^{N}}\left|u(x)- u(y)\right|^{q-2}\dfrac{\left(u(x)-u(y)\right)\left(v(x)-v(y)\right)}{\left|x-y\right|^{N+\beta q}}dxdy = \lambda\displaystyle\int_{\R^N}a(x) \left|u\right|^{p-2}u vdx
\end{align*}
for all $v \in \CW$. 
The following proposition is a result on existence eigenvalue of \eqref{generality}.
\begin{proposition}\label{proposition_4.1}
Suppose that \eqref{1.11} holds. Then the quantity
  \begin{align}\label{eigenvalue}
 \lambda_1\left(\alpha, \beta, p, q,\, \R^N\right) = \inf\limits_{u\in\CW \setminus \left\{0\right\}}\dfrac{\displaystyle\int_{\R^N}\displaystyle\int_{\R^N}\dfrac{\left|u(x) - u(y)\right|^p}{\left|x-y\right|^{N+\alpha p}}dxdy +\dfrac{p}{q}\displaystyle\int_{\R^N}\displaystyle\int_{\R^N}\dfrac{\left|u(x) - u(y)\right|^q}{\left|x-y\right|^{N+\beta q}}dxdy}{\displaystyle\int_{\R^N}a \left|u\right|^{p}dx}
 \end{align}
 is positive principal eigenvalue $\lambda_1\left(\alpha, \beta, p, q, \, \R^N\right)\in \R$, which is simple and the associated positive eigenfunction.
\end{proposition}
\begin{proof}[\bf Proof]

We integrate the ideas of [\cite{BZ}, Proposition 1] to deal with existence eigenvalue in problem \eqref{generality}.
First, we define $\Phi: \CW\setminus\left\{0\right\}\to \R^+$ by
\begin{align*}
\Phi(u) = \dfrac{\displaystyle\int_{\R^N}\displaystyle\int_{\R^N}\dfrac{\left|u(x) - u(y)\right|^p}{\left|x-y\right|^{N+\alpha p}}dxdy +\dfrac{p}{q}\displaystyle\int_{\R^N}\displaystyle\int_{\R^N}\dfrac{\left|u(x) - u(y)\right|^q}{\left|x-y\right|^{N+\beta q}}dxdy}{\displaystyle\int_{\R^N}a(x) \left|u\right|^{p}dx}.
\end{align*}
It is easy check that $\Phi$ is well-defined since $\displaystyle\int_{\R^N}a(x) \left|u\right|^{p}dx \leq \left\|a\right\|_{L^{\infty}\left(\R^N\right)}\left\|u\right\|_{L^p\left(\R^N\right)}^p<+\infty$. We recall that the first eigenvalue $\lambda_1$ of the fractional $p-$Laplacian without weight function  is characterized by
\begin{align*}
\lambda_1 = \inf\limits_{u\in W^{\alpha, p}\left(\R\right)\setminus\left\{0\right\}}\dfrac{\displaystyle\int_{\R^N}\displaystyle\int_{\R^N}\dfrac{\left|u(x) - u(y)\right|^p}{\left|x-y\right|^{N+\alpha p}}dxdy }{\displaystyle\int_{\R^N} \left|u\right|^{p}dx}.
\end{align*}
It is noticed that 
\begin{align*}
\dfrac{\displaystyle\int_{\R^N}\displaystyle\int_{\R^N}\dfrac{\left|u(x) - u(y)\right|^p}{\left|x-y\right|^{N+\alpha p}}dxdy }{\displaystyle\int_{\R^N} \left|a(x)\right|\left|u\right|^{p}dx}\geq \dfrac{1}{\left\|a\right\|_{L^{\infty}\left(\R^N\right)}}\dfrac{\displaystyle\int_{\R^N}\displaystyle\int_{\R^N}\dfrac{\left|u(x) - u(y)\right|^p}{\left|x-y\right|^{N+\alpha p}}dxdy }{\displaystyle\int_{\R^N} \left|u\right|^{p}dx}\geq \lambda_1.
\end{align*}
This implies $\lambda_1\left(\alpha, \beta, p, q,\, \R^N\right)\geq \lambda_1>\cdots >-\infty$. Thus, $\lambda_1\left(\alpha, \beta, p, q,\, \R^N \right)$ exists. 

We claim that $\lambda_1\left(\alpha, \beta, p, q, \, \R^N\right) = \inf\limits_{u\in\CW\setminus\left\{0\right\}}\Phi(u)$. Let $\left\{u_j\right\}\subseteq \CW\setminus\left\{0\right\}$ such that $\lim\limits_{j\to+\infty} \Phi\left(u_j\right) =\lambda_1\left(\alpha, \beta, p, q, \, \R^N\right)$ in $\R$, we deduce that there exists a $\kappa(\varepsilon)\in \N$ such that  $0<\Phi\left(u_j\right)-\lambda_1\left(\alpha, \beta, p, q, \, \R^N\right)<\varepsilon$ for every $j\geq \kappa(\varepsilon)$, within $\varepsilon>0$.
 Since $\left|\left|u_j(x)\right| -\left|u_j(y)\right| \right| \leq \left|u_j(x)- u_j(y)\right|,\,\,\forall \left(x,y\right)\in \R^{2N}$, we have
 \begin{align*}
 0<\Phi\left(\left|u_j\right|\right)-\lambda_1\left(\alpha, \beta, p, q,\, \R^N\right)<\varepsilon,\,\, \forall j\geq \kappa\left(\varepsilon\right).
 \end{align*}
Observe that $\left|u_j\right| \in \CW$ because $u_j \in \CW$.  We may assume, without loss of generality, that $u_j\geq 0$ (we only consider $\left|u_j\right|$). We obtain that $\left\{\Phi\left(u_j\right)\right\}$  is a convergent sequence in $\R$ and then it is also bounded, namely $\left|\Phi\left(u_j\right)\right|\leq C,\,\,\,\forall j\geq 1$. Since $a(x)$ is bounded, we deduce from \eqref{1.11} that
\begin{align}\label{5.1}
\displaystyle\int_{\R^N} \left|a(x)\right|\left|u_j\right|^{p}dx &\leq \displaystyle\int_{\R^N} \left|a(x)\right|^{\frac{pt}{p_{\alpha}^*}}\left|u_j\right|^{tp}\left|a(x)\right|^{\frac{p\left(1-t\right)}{s}}\left|u_j\right|^{\left(1-t\right)p}dx\nonumber\\& \leq
\left(\displaystyle\int_{\R^N}\left|a(x)\right|\left|u_j\right|^{p_{\alpha}^*}dx\right)^{\frac{pt}{p_{\alpha}^*}}\left(\displaystyle\int_{\R^N}a(x)\left|u_j\right|^{s}dx\right)^{\frac{p\left(1-t\right)}{s}}.
\end{align}
It follows from \eqref{5.1} that
 \begin{align*}
 \displaystyle\int_{\R^N}\left|a(x)\right|\left|u_j\right|^{p_{\alpha}^*}dx\leq \left\|a\right\|_{L^{\infty}\left(\O\right)}\displaystyle\int_{\R^N}\left|u_j(x)\right|^{p_{\alpha}^*}dx\leq C\left\|u_j\right\|^{pt}_{W^{\alpha, p}\left(\R^N\right)}\left\|u_j\right\|^{p\left(1-t\right)}_{W^{\beta, q}\left(\R^N\right)}\,\,\text{with}\,\,\ C>0.
 \end{align*}
 Therefore, we have
\begin{align}\label{3.31}
\Phi\left(u_j\right)&\geq \dfrac{\displaystyle\int_{\R^N}\displaystyle\int_{\R^N}\dfrac{\left|u_j(x) - u_j(y)\right|^p}{\left|x-y\right|^{N+\alpha p}}dxdy +\dfrac{p}{q}\displaystyle\int_{\R^N}\displaystyle\int_{\R^N}\dfrac{\left|u_j(x) - u_j(y)\right|^q}{\left|x-y\right|^{N+\beta q}}dxdy}{C\left\|u_j\right\|^{pt}_{W^{\alpha, p}\left(\R^N\right)}\left\|u_j\right\|^{p\left(1-t\right)}_{W^{\beta, q}\left(\R^N\right)}}\nonumber\\&\geq
C_1\dfrac{\left\|u_j\right\|^p_{W^{\alpha, p}\left(\R^N\right)}+ \left\|u_j\right\|^q_{W^{\beta, q}\left(\R^N\right)}}{\left\|u_j\right\|^{pt}_{W^{\alpha, p}\left(\R^N\right)}\left\|u_j\right\|^{p\left(1-t\right)}_{W^{\beta, q}\left(\R^N\right)}}.
\end{align}
Note that $\left\|u\right\|_{\CW} = \left\|u\right\|_{\CW^{\alpha, p}\left(\R^N\right)}+ \left\|u\right\|_{\CW^{\beta, q}\left(\R^N\right)}$. Assume that $\left\{u_j\right\}$ is not bounded, so for every $M>0$, there exists a $j_M\geq M$ such that $\left\|u_{j_M}\right\|_{\CW}\geq M$. Consequently, there exists a subsequence $\left\{u_{j_k}\right\}$ satisfies $\lim\limits_{k\to+\infty}\left\|u_{j_k}\right\|_{\CW} =+\infty$ in $\R$, which together with \eqref{3.31} yields
\begin{align}\label{3.031}
\Phi\left(u_{j_k}\right)\geq C_2\left(\dfrac{\left\|u_{j_k}\right\|^{p\left(1-t\right)}_{W^{\alpha, p}\left(\R^N\right)}}{\left\|u_{j_k}\right\|^{p\left(1-t\right)}_{W^{\beta, q}\left(\R^N\right)}}+ \dfrac{\left\|u_{j_k}\right\|^{q-p\left(1-t\right)}_{W^{\beta, q}\left(\R^N\right)}}{\left\|u_{j_k}\right\|^{pt}_{W^{\alpha, p}\left(\R^N\right)}}\right),\,\,\,\forall j_k \geq k\geq 1.
\end{align}
Since $\left\|u_{j_k}\right\|_{W^{\alpha, p}\left(\R^N\right)} >\varepsilon,\,\,\,\forall j_k\geq k\geq\kappa\left(\varepsilon\right)$, we deduce
\begin{align}\label{3.032}
\Phi\left(u_{j_k}\right)\geq \left\{\begin{array}{ll}
C_3\left(\dfrac{\varepsilon^{p\left(1-t\right)}}{\left\|u_{j_{\kappa\left(\varepsilon\right)}}\right\|^{p\left(1-t\right)}_{W^{\beta, q}\left(\R^N\right)}}+ \dfrac{\varepsilon^{q-p\left(1-t\right)}}{\left\|u_{j_{\kappa\left(\varepsilon\right)}}\right\|^{pt}_{W^{\alpha, p}\left(\R^N\right)}}\right), \, &\mbox{if } q-p\left(1-t\right)\geq 0.  \smallskip\\
C_4\left(\dfrac{\varepsilon^{p\left(1-t\right)}}{\left\|u_{j_{\kappa\left(\varepsilon\right)}}\right\|^{p\left(1-t\right)}_{W^{\beta, q}\left(\R^N\right)}}+ \dfrac{1}{\left\|u_{j_{\kappa\left(\varepsilon\right)}}\right\|^{pt}_{W^{\alpha, p}\left(\R^N\right)}}\dfrac{1}{\left\|u_{j_{\kappa\left(\varepsilon\right)}}\right\|^{p\left(1-t\right)-q}_{W^{\alpha, p}\left(\R^N\right)}}\right),&\mbox{if } q-p\left(1-t\right)< 0.
\end{array}\right.
\end{align}
Combining \eqref{3.031}, \eqref{3.032} and passing to the limit as $k\to+\infty$, we derive a contradiction. This confirms that $\left\{u_j\right\}$ is bounded in $\R^N$. Obviously, the space $\CW = W^{\alpha, p}\left(\R^N\right)\bigcap W^{\beta, q}\left(\R^N\right)$ is reflexive, then there exist a $\overline{u}\in \CW,\, \overline{u}\geq  0$ such that  $u_j$ weakly converge to $\overline{u}$ in $\CW^*\bigcap L^p\left(\R^N\right)$ ($\CW^*$ is a dual of $\CW$). Next, we consider a linear functional 
$T: W^{\alpha, p}\left(\R^N\right)\to \R$ with $T(v)(x) = \left(-\Delta_p\right)^{\alpha}v(x)\varphi(x)$, for all $\varphi \in C_c^{\infty}\left(\R^N\right)$. Similarly, we also consider a linear function $\widetilde{T}: W^{\beta, q}\left(\R^N\right)\to \R$ with $\widetilde{T}(u)(x) = \left(-\Delta_q\right)^{\beta}u(x)\varphi(x)$, for all $\varphi \in C_c^{\infty}\left(\R^N\right)$. Thus, there exists a linear functional $\Gamma: \CW\to \R$ such that
\begin{align*}
\Gamma = T_{|\CW} = \widetilde{T}_{|\CW}.
\end{align*}
Since $u_j$ weakly converge to $\overline{u}$, we get $\lim\limits_{j\to+\infty}\Gamma\left(u_j\right) = \Gamma\left(\overline{u}\right)$ in $\CW$. Applying standard Fatou's Lemma, we can see that
\begin{align*}
&\left|\lim\limits_{m\to+\infty}\inf\limits_{j\geq m}\displaystyle\int_{\R^N}\displaystyle\int_{\R^N}\dfrac{\left|u_j(x) - u_j(y)\right|^{p-2}\left(u_j(x)-u_j(y)\right)}{\left|x-y\right|^{N+\alpha p}}\left(\varphi(x)-\varphi(y)\right)dxdy\right|\\&\geq 
\left|\displaystyle\int_{\R^N}\displaystyle\int_{\R^N}\lim\limits_{m\to+\infty}\inf\limits_{j\geq m}\dfrac{\left|u_j(x) - u_j(y)\right|^{p-2}\left(u_j(x)-u_j(y)\right)}{\left|x-y\right|^{N+\alpha p}}\left(\varphi(x)-\varphi(y)\right)dxdy\right|\\& \geq
\left|\displaystyle\int_{\R^N}\displaystyle\int_{\R^N}\dfrac{\left|\overline{u}(x) - \overline{u}(y)\right|^{p-2}\left(\overline{u}(x)-\overline{u}(y)\right)}{\left|x-y\right|^{N+\alpha p}}\left(\varphi(x)-\varphi(y)\right)dxdy\right|.
\end{align*}
Hence
\begin{align*}
\left\|T\left(\overline{u}\right)\right\|_{W^{\alpha, p}\left(\R^N\right)}\leq C_4 \lim\limits_{m\to+\infty}\inf\limits_{j\geq m}\left\|\varphi\right\|_{W^{\alpha, p}\left(\R^N\right)}\left\|u\right\|^{p-1}_{W^{\alpha, p}\left(\R^N\right)}.
\end{align*}
By the properties of dual space, then for all $\varphi\in C_c^{\infty}\left(\R^N\right)$, we have 
\begin{align*}
\left[\overline{u}\right]_{\alpha, p}\leq  C_4 \lim\limits_{m\to+\infty}\inf\limits_{j\geq m}\left\|\varphi\right\|_{W^{\alpha, p}\left(\R^N\right)}\left\|u\right\|^{p-1}_{W^{\alpha, p}\left(\R^N\right)}.
\end{align*}
We can chose $\varphi$ satisfying $C_4\left\|\varphi\right\|_{W^{\alpha, p}\left(\R^N\right)}<1$, and
\begin{align*}
\left[\overline{u}\right]_{\alpha, p}\leq   \lim\limits_{m\to+\infty}\inf\limits_{j\geq m}\left[u_j\right]_{\alpha, p}.
\end{align*}
Similarly, $\left[\overline{u}\right]_{\beta, q}\leq   \lim\limits_{m\to+\infty}\inf\limits_{j\geq m}\left[u_j\right]_{\beta, q}$.
We have $\lim\limits_{j\to\infty} u_j = \overline{u}$ in $L^p\left(\R^N\right)$, and $\lim\limits_{j\to\infty} u_j = \overline{u}$ a.e in $\R^N$. Hence, 
\begin{align*}
\lim\limits_{j\to\infty}\displaystyle\int_{\R^N}a(x)\left|u_j\right|^pdx = \displaystyle\int_{\R^N}a(x)\left|\overline{u}\right|^pdx,
\end{align*}
which implies
\begin{align}
\lim\limits_{m\to+\infty}\inf\limits_{j\geq m}\Phi\left(u_j\right)&= \dfrac{\displaystyle\int_{\R^N}\displaystyle\int_{\R^N}\dfrac{\left|u_j(x) - u_j(y)\right|^p}{\left|x-y\right|^{N+\alpha p}}dxdy +\dfrac{p}{q}\displaystyle\int_{\R^N}\displaystyle\int_{\R^N}\dfrac{\left|u_j(x) - u_j(y)\right|^q}{\left|x-y\right|^{N+\beta q}}dxdy}{\displaystyle\int_{\R^N}a(x)\left|u_j\right|^pdx}\nonumber\\& \geq
\dfrac{\lim\limits_{m\to+\infty}\inf\limits_{j\geq m}\displaystyle\int_{\R^N}\displaystyle\int_{\R^N}\dfrac{\left|u_j(x) - u_j(y)\right|^p}{\left|x-y\right|^{N+\alpha p}}dxdy +\dfrac{p}{q}\lim\limits_{m\to+\infty}\inf\limits_{j\geq m}\displaystyle\int_{\R^N}\displaystyle\int_{\R^N}\dfrac{\left|u_j(x) - u_j(y)\right|^q}{\left|x-y\right|^{N+\beta q}}dxdy}{\lim\limits_{m\to+\infty}\sup\limits_{j\geq m}\displaystyle\int_{\R^N}a(x)\left|u_j\right|^pdx}\nonumber\\&\geq
\dfrac{\displaystyle\int_{\R^N}\displaystyle\int_{\R^N}\dfrac{\left|\overline{u}(x) - \overline{u}(y)\right|^p}{\left|x-y\right|^{N+\alpha p}}dxdy +\dfrac{p}{q}\displaystyle\int_{\R^N}\displaystyle\int_{\R^N}\dfrac{\left|\overline{u}(x) - \overline{u}(y)\right|^q}{\left|x-y\right|^{N+\beta q}}dxdy}{\displaystyle\int_{\R^N}a(x)\left|\overline{u}\right|^pdx}.
\end{align}
Therefore, we deduce that $\lambda_1\left(\alpha, \beta, p, q,\, \R^N\right)\geq \lim\limits_{m\to+\infty}\inf\limits_{j\geq m}\Phi\left(u_j\right)\geq \Phi\left(\overline{u}\right)\geq \lambda_1\left(\alpha, \beta, p, q,\, \R^N\right)$ since $\overline{u}\in \CW$. Hence we infer that $\lambda_1\left(\alpha, \beta, p, q,\, \R^N\right) =\Phi\left(\overline{u}\right)$  and consequently we conclude that $\overline{u}$ is a critical point of $\Phi$, i.e.
\begin{align}\label{3.32}
\Phi^{\prime}\left(\overline{u}\right)(v) = 0,\,\,\,\forall v\in \CW\setminus\left\{0\right\}.
\end{align}
Setting $\Upsilon\left(u\right) = \displaystyle\int_{\R^N}\displaystyle\int_{\R^N} \dfrac{\left|u(x)-u(y)\right|^p}{|x-y|^{N+\alpha p}}dxdy$. It is easy to check that 
\begin{align*}
\Upsilon^{\prime}\left(u\right)(v) = p\displaystyle\int_{\R^N}\displaystyle\int_{\R^N} \dfrac{\left|u(x)-u(y)\right|^{p-2}\left(u(x) - u(y)\right)\left(v(x)-v(y)\right)}{|x-y|^{N+\alpha p}}dxdy.
\end{align*}
Similar, 
\begin{align*}
\widetilde{\Upsilon }^{\prime}\left(u\right)(v) = q\displaystyle\int_{\R^N}\displaystyle\int_{\R^N} \dfrac{\left|u(x)-u(y)\right|^{q-2}\left(u(x) - u(y)\right)\left(v(x)-v(y)\right)}{|x-y|^{N+\beta q}}dxdy.
\end{align*}
Furthermore, we also have $\Theta^{\prime}(u)(v) = p\displaystyle\int_{\R^N}a(x)\left|u\right|^{p-2}u(x)v(x)dx$, within $\Theta(u)= p\displaystyle\int_{\R^N}a(x)\left|u\right|^{p}dx$. Therefore, the equation \eqref{3.32} is equivalent to
\begin{align}
\Phi^{\prime}\left(\overline{u}\right)(v) =  \dfrac{\left[1/p\Upsilon^{\prime}\left(\overline{u}\right)(v)+ 1/q \widetilde{\Upsilon }^{\prime}\left(\overline{u}\right)(v)\right]1/p\Theta (\overline{u}) -\Theta^{\prime}(u)(v)\left(1/p\Upsilon\left(\overline{u}\right)+1/q\widetilde{\Upsilon }\left(\overline{u}\right) \right) }{\left(1/p\Theta (\overline{u})\right)^2}.
\end{align}
This together with \eqref{3.32}, yields
\begin{align*}
\dfrac{1}{p}\Upsilon^{\prime}\left(\overline{u}\right)(v)+ \dfrac{1}{q}\widetilde{\Upsilon }\left(\overline{u}\right)(v) = \lambda_1\left(\alpha, \beta, p, q,\, \R^N\right)\dfrac{ \Theta^{\prime}(\overline{u})(v)}{p}.
\end{align*}
This means that $ \lambda_1\left(\alpha, \beta, p, q, \, \R^N\right)$ is an eigenvalue of \eqref{generality} associated to eigenfunction $\overline{u}\geq 0$. We claim that $\overline{u}\neq0$. Arguing indirectly we assume that this does not hold, and hence, $\overline{u} = 0$. Thus, we have that $\lim\limits_{j\to+\infty}\displaystyle\int_{\R^{N}}a(x)\left|u_j\right|^pdx = \displaystyle\int_{\R^{N}}a(x)\left|\overline{u}\right|^pdx=0$. Consequently, we have
\begin{align*}
\lim\limits_{j\to+\infty}\left(\dfrac{1}{p}\Upsilon\left(u_j\right)+ \dfrac{1}{q}\widetilde{\Upsilon}\left(u_j\right)\right) = \lim\limits_{j\to+\infty}\Phi\left(u_j\right)\dfrac{\Theta\left(u_j\right)}{p}=0,
\end{align*}
which implies
\begin{align}\label{3.34}
\lim\limits_{j\to+\infty}\left[u_j\right]_{W^{\alpha, p}\left(\R^N\right)}=\lim\limits_{j\to+\infty}\left[u_j\right]_{W^{\beta, q}\left(\R^N\right)}=0,
\end{align}
and
\begin{align}\label{3.35}
&\left\|u_j\right\|_{W^{\alpha, p}\left(\R^N\right)}\leq C_p\left[u_j\right]_{W^{\alpha, p}\left(\R^N\right)},\nonumber\\&
\left\|u_j\right\|_{W^{\beta, q}\left(\R^N\right)}\leq C_q\left[u_j\right]_{W^{\beta, q}\left(\R^N\right)}.
\end{align}
Combining Eqs. \eqref{3.34} and \eqref{3.35}, we deduce $\lim\limits_{j\to+\infty}\left\|u_j\right\|_{W^{\alpha, p}\left(\R^N\right)}= \lim\limits_{j\to+\infty}\left\|u_j\right\|_{W^{\beta, q}\left(\R^N\right)} =0.$
Consequently,
\begin{align*}
\left\|u_j\right\|^{p}_{W^{\beta, q}\left(\R^N\right)}\leq \left\|u_j\right\|^{p}_{W^{\alpha, p}\left(\R^N\right)}\left\|u_j\right\|^{\frac{p-q}{t}}_{W^{\beta, q}\left(\R^N\right)},\,\,\,\forall k\geq \kappa(j),
\end{align*}
where $\varepsilon = \left\|u_j\right\|^{p}_{\CW^{\alpha, p}\left(\R^N\right)}\left\|u_j\right\|^{\frac{p-q}{t}}_{\CW^{\beta, q}\left(\R^N\right)}$. Passing subsequence and combining \eqref{3.31}, we get
\begin{align}
\Phi\left(u_j\right)&\geq C_1 \dfrac{\left\|u_j\right\|^{p-\frac{p-q}{t}}_{W^{\alpha, p}\left(\R^N\right)}+ \left\|u_j\right\|^q_{W^{\beta, q}\left(\R^N\right)}}{\left\|u_j\right\|^{pt}_{W^{\alpha, p}\left(\R^N\right)}\left\|u_j\right\|^{p\left(1-t\right)}_{W^{\beta, q}\left(\R^N\right)}}\\&\geq C_1\dfrac{1}{\left\|u_j\right\|^{pt}_{W^{\alpha, p}\left(\R^N\right)}} \dfrac{1}{\left\|u_j\right\|^{p-q-pt^2}_{W^{\beta, q}\left(\R^N\right)}}\left(1+\left\|u_j\right\|^{\frac{-\left(1-t\right)\left(q-p\right)}{t}}_{W^{\beta, q}\left(\R^N\right)}\right).
\end{align}
Since $t<\sqrt{\dfrac{p-q}{p}}$, we deduce $p-q-t^2p>0$. This yields that $\lim\limits_{j\to+\infty}\Phi\left(u_j\right) = +\infty$ in $\R$. We get a contradiction. This shows that $\overline{u}\neq 0$ and $\lambda_1\left(\alpha, \beta, p, q,\, \R^N\right)$ is a principal eigenvalue of problem \eqref{generality}.

Now, we prove that $\lambda_1\left(\alpha,\beta, p, q,\, \R^N\right)$ is simple. By the proof above, we may assume that $u > 0$ and $v > 0$. Our claim is that $u(x) = Cv(x)$ with $C>0$.
Normalize the functions so that
\begin{align}\label{normalize}
\displaystyle\int_{\O}a(x) u^{p}dx =\displaystyle\int_{\O}a(x)v^{p}dx=1
\end{align}
and consider the admissible function
\begin{align*}
w_1 = \left(\dfrac{u^p+v^p}{2}\right)^{1/p}, w_2 = \left(\dfrac{u^q+v^q}{2}\right)^{1/q}.
\end{align*}
Thanks to the inequalities
\begin{align}\label{2.032}
\left|w_i(x) - w_i(y)\right|^p\leq \dfrac{1}{2}\left|u(x)-u(y)\right|^p + \dfrac{1}{2}\left|v(x)-v(y)\right|^p,\,\,\ i=1,2
\end{align}
and incorporating the eigenvalue of problem\eqref{generality}, we obtain
\begin{align}\label{3.28}
\lambda_1\left(\alpha,\beta, p, q,\, \R^N\right) \leq  \dfrac{1}{2}\dfrac{\Gamma\left(J_p\left(w_i(x), w_i(y)\right)\right)}{\displaystyle\int_{\R^N}a(x) \left|w_1\right|^{p}dx} + \dfrac{1}{2}\dfrac{\Gamma\left(J_q\left(w_2(x), w_2(y)\right)\right)}{\displaystyle\int_{\R^N}a(x) \left|w_2\right|^{q}dx},
\end{align}
where $J_r\left(w_i(x), w_i(y)\right) = \left|w_i(x)^{\frac{1}{r}}- w_i(y)^{\frac{1}{r}}\right|^r$,  $r = p, q$, $s =  \alpha ,\beta$, $i=1,2$, and 
\begin{align*}
\Gamma\left(J_r\left(w_i(x), w_i(y)\right)\right) =  \dfrac{\displaystyle\int_{\R^{N}}\displaystyle\int_{\R^{N}}J_r\left(w_i(x), w_i(y)\right)dxdy}{\left|x-y\right|^{N+s r}}.
\end{align*}
Using the fact that $\displaystyle\int_{\O} w_i^{p}dx  = 1$, within $i=1,2$. This together with \eqref{normalize}, \eqref{3.28} yields that
\begin{align*}
\lambda_1\left(\alpha,\beta, p, q,\, \R^N\right) \leq \dfrac{1}{2}\lambda_1\left(\alpha,\beta, p, q,\, \R^N\right)+ \dfrac{1}{2}\lambda_1\left(\alpha,\beta, p, q,\, \R^N\right) = \lambda_1\left(\alpha,\beta, p, q,\, \R^N\right).
\end{align*}
Thus, we proved that $u(x) = C v(x)$ with $C\in\R$, as claimed.
\end{proof}

\subsection{ A continuous family of eigenvalues}
Let $I\left(u\right):\CW\to\R$ be the energy functional associated to Eq. \eqref{generality} deﬁned by
\begin{align}\label{5.14}
I\left(u\right) &= \dfrac{1}{p}\displaystyle\int_{\R^N}\displaystyle\int_{\R^N}\dfrac{\left|u(x) - u(y)\right|^p}{\left|x-y\right|^{N+\alpha p}}dxdy +\dfrac{1}{q}\displaystyle\int_{\R^N}\displaystyle\int_{\R^N}\dfrac{\left|u(x) - u(y)\right|^q}{\left|x-y\right|^{N+\beta q}}dxdy- \dfrac{\lambda}{p}\displaystyle\int_{\R^N}a(x)\left|u(x)\right|^pdx.
\end{align}
From the embedding inequalities and assumptions \eqref{1.11}, we see that the functional $I$ is well deﬁned and $I\in C^1\left(\CW, \R\right)$ with
\begin{align}\label{5.15}
\left(I^{\prime}(u),\varphi\right) &=  \displaystyle\int_{\R^N}\displaystyle\int_{\R^N}\left[\dfrac{\left|u(x) - u(y)\right|^{p-2}}{\left|x-y\right|^{N+\alpha p}} + \dfrac{\left|u(x) - u(y)\right|^{q-2}}{\left|x-y\right|^{N+\beta q}}\right]\left(u(x)-u(y)\right)\left(\varphi(x)-  \varphi(y)\right)dxdy \nonumber\\& \hspace{8cm}-\lambda\displaystyle\int_{\R^N}a(x)\left|u\right|^{p-2} u\varphi dx.
\end{align}
The results of this subsection are stated in the following propositions.
\begin{proposition}\label{proposition_4.2}
Assume that \eqref{1.11} holds. For any $\lambda> \lambda_1\left(\alpha, \beta, p, q\right)$ is an eigenvalue of problem \eqref{generality}  then there exist a continuous set of positive eigenvalues.
\end{proposition}
The idea of the proof is inspired by [\cite{BZ1}, Theorem 1]. The proof of proposition \ref{proposition_4.2} will be divided into several lemmas. We first state the following Palais-Smale condition.
\begin{lemma}\label{lem_4.3}
Suppose that \eqref{1.11} holds. The functional $I$ satisfies the Palais-Smale condition $\left(PS\right)_c$ for any $c\in\R$.
\end{lemma}
\begin{proof}[\bf Proof]
Let $\left\{u_j\right\}\subset \CW$ be a $(PS)_c$ sequence of $I$. Therefore,
\begin{align}
I\left(u_j\right) = c+ o(1),\,\,\, I^{\prime}\left(u_j\right) = o(1).
\end{align}
Then, we have
\begin{align}\label{5.16}
I\left(u_j\right) &= \dfrac{1}{p}\displaystyle\int_{\R^N}\displaystyle\int_{\R^N}\dfrac{\left|u_j(x) - u_j(y)\right|^p}{\left|x-y\right|^{N+\alpha p}}dxdy +\dfrac{1}{q}\displaystyle\int_{\R^N}\displaystyle\int_{\R^N}\dfrac{\left|u_j(x) - u_j(y)\right|^q}{\left|x-y\right|^{N+\beta q}}dxdy- \dfrac{\lambda}{p}\displaystyle\int_{\R^N}a(x)\left|u_j(x)\right|^pdx = c+ o(1)
\end{align}
and 
\begin{align}\label{5.016}
\left(I^{\prime}\left(u_j\right), u_j\right)& = \displaystyle\int_{\R^N}\displaystyle\int_{\R^N}\left[\dfrac{\left|u_j(x) - u_j(y)\right|^{p-2}}{\left|x-y\right|^{N+\alpha p}} + \dfrac{\left|u_j(x) - u_j(y)\right|^{q-2}}{\left|x-y\right|^{N+\beta q}}\right]\left(u_j(x)-u_j(y)\right)\left(u_j(x)-u_j(y)\right)dxdy \\& \hspace{8cm}-\lambda\displaystyle\int_{\R^N}a(x)\left|u_j\right|^{p-2} uv dx.
\end{align}
So
\begin{align}\label{5.17}
o(1)\left\|u_j\right\|_{\CW} +o(1)& \geq I\left(u_j\right)-\dfrac{1}{p}\left(I^{\prime}\left(u_j\right), u_j\right)\nonumber\\&
= \dfrac{1}{p}\left[u_j\right]_{\alpha, p} + \dfrac{1}{q}\left[u_j\right]_{\beta, q}-\dfrac{\lambda}{p}\displaystyle\int_{\R^N}a(x)\left|u_j(x)\right|^pdx\nonumber\\&  - \left[\dfrac{1}{p}\left[u_j\right]_{\alpha, p} - \left(\dfrac{1}{q}-\dfrac{1}{p}\right)\left[u_j\right]_{\beta, q}-\dfrac{\lambda}{p}\displaystyle\int_{\R^N}a(x)\left|u_j(x)\right|^pdx\right]\nonumber\\&
= \left(\dfrac{1}{q}-\dfrac{1}{p}\right)\left[u_j\right]_{\beta, q}.
\end{align}
Combining \eqref{5.16}, \eqref{5.17} , we infer that $\left\{u_j\right\}$ is bounded in $W^{\beta, q}\left(\R^N\right)$. Next, we set
\begin{align*}
\widetilde{I}\left(u_j\right) = \displaystyle\int_{\R^N}\displaystyle\int_{\R^N}\dfrac{\left|u_j(x) - u_j(y)\right|^p}{\left|x-y\right|^{N+\alpha p}}dxdy- \lambda\displaystyle\int_{\R^N}a(x)\left|u_j\right|^{p} dx.
\end{align*}
It is easy to check that $\widetilde{I}\left(u_j\right)$ is a real convergent sequence. So, there
exists $M \geq 0$ such that
\begin{align}\label{4.23}
\displaystyle\int_{\R^N}\displaystyle\int_{\R^N}\dfrac{\left|u_j(x) - u_j(y)\right|^p}{\left|x-y\right|^{N+\alpha p}}dxdy- \lambda\displaystyle\int_{\R^N}a(x)\left|u_j\right|^{p} dx \leq M.
\end{align}
It follows by \eqref{5.1} and $\left\{u_j\right\}$ is bounded in $W^{\beta, q}\left(\R^N\right)$, we can find a positive constant $L$ such that
\begin{align}\label{24}
\left\|u_j\right\|_{W^{\alpha, p}\left(\R^N\right)}^{p}\leq L\left(1 +M \left\|u_j\right\|_{W^{\alpha, p}\left(\R^N\right)}^{pt}\right).
\end{align}
Since $\left\{\widetilde{I}\left(u_j\right)\right\}$ is bounded and \eqref{24} holds, this implies that $\left\{u_j\right\}$ is bounded in $W^{\alpha, p}\left(\R^N\right)$ . Hence we infer that $\left\{u_j\right\}$  is bounded in $\CW$. Consequently, up to a subsequence,  there exists $u\in \CW$ 
such that $u_j \rightharpoonup  u$ in $\CW$ and $u_n\to u$ 
in $L^r_{loc}\left(\R^N\right)$ for any $r\in\left[1, 
p_{\alpha}^*\right)$. Now, we aim to prove the strong 
convergence of $u_j$ to $u$ in $\CW$.
Setting $p^{\prime}=p/(p-1)$ and using the Hölder conjugate of $p$, we observe that the sequence
 \begin{align*}
 \dfrac{\left|u_j(x)-u_j(y)\right|^{p-2}\left(u_j(x)-u_j(y)\right)}{\left|x-y\right|^{\frac{N+\alpha p}{p'}}}
 \end{align*}
 is bounded in $L^{p^{\prime}}\left(\R^{2N}\right)$, and $\dfrac{\varphi(x)-\varphi(y)}{\left|x-y\right|^{\frac{N+\alpha p}{p}}}\in L^{p}\left(\R^{2N}\right)$. Thus,
 \begin{align*}
 \dfrac{\left|u_j(x)-u_j(y)\right|^{p-2}\left(u_j(x)-u_j(y)\right)\left(\varphi(x)-\varphi(y)\right)}{\left|x-y\right|^{\frac{N+\alpha p}{p'}}} \to \dfrac{\left|u(x)-u(y)\right|^{p-2}\left(u(x)-u(y)\right)\left(\varphi(x)-\varphi(y)\right)}{\left|x-y\right|^{N+\alpha p}}\,\,\text{a.e\, in}\,\, \R^{2N}.
 \end{align*}
 Similarly, the second integral in \eqref{5.15} converges to
 \begin{align*}
  \dfrac{\left|u(x)-u(y)\right|^{q-2}\left(u(x)-u(y)\right)\left(\varphi(x)-\varphi(y)\right)}{\left|x-y\right|^{N+\beta q}}.
 \end{align*}
 Furthermore
 \begin{align*}
 \displaystyle\int_{\R^N}a(x)\left|u_j\right|^{p-2} u_j\varphi dx \to \displaystyle\int_{\R^N}a(x)\left|u\right|^{p-2} u\varphi dx.
 \end{align*}
Then passing to the limit in \eqref{5.15} shows that $u\in \CW$ is a weak solution of \eqref{generality}. Setting $v_j= u_j-u$, we will show that $v_j\to 0$ in $\CW$. Applying Brezis-Lieb lemma, Lemma \ref{lem.2.1}, we can obtain 
\begin{align*}
\left\|u_j\right\|_{W^{\alpha, p}\left(\R^N\right)}^p = \left\|v_j\right\|_{W^{\alpha, p}\left(\R^N\right)}^p + \left\|u\right\|_{W^{\alpha, p}\left(\R^N\right)}^p +o(1).
\end{align*}
Similarly, we have
\begin{align*}
\left\|u_j\right\|_{W^{\beta, q}\left(\R^N\right)}^q = \left\|v_j\right\|_{W^{\beta, q}\left(\R^N\right)}^q + \left\|u\right\|_{W^{\beta, q}\left(\R^N\right)}^q +o(1).
\end{align*}
Taking $\varphi=u_j$. Consequently, $u_j\to u$ in $\CW$ as $j\to \infty$. This ends the proof of lemma.
\end{proof}
Next, we will show that the functional $I$ given by \eqref{5.14} satisfies the Mountain
pass geometry.
\begin{lemma}\label{lem_4.4}
The functional $I$ satisﬁes the following conditions:
\begin{enumerate}
\item there exist $\tau ,\mu  > 0$ such that $I(u)\geq \tau$ with $\left\|u\right\|_{\CW} = \mu$,
\item there exists $u_0\in \CW$ with $\left\|u_0\right\|_{\CW}> \mu$ such that $I\left(u_0\right)<0$.
\end{enumerate}
\end{lemma}
\begin{proof}[\bf Proof]
It is notice that $\left\|\cdot\right\|_{W^{\alpha, p}\left(\R^N\right)}$ and $\left[\cdot\right]_{\alpha, p}$ are equivalent. It follows from \eqref{5.1}, we have
\begin{align}
\displaystyle\int_{\R^N} \left|a(x)\right|\left|u_j\right|^{p}dx &\leq C\left\|a\right\|_{L^{+\infty}\left(\R^N\right)}^{\frac{pt}{p_{\alpha}^*}}\left\|a\right\|_{L^{\left(\frac{q_{\beta}^*}{s}\right)'}\left(\R^N\right)}^{\frac{p(1-t)}{s}}\left[u\right]_{\alpha, p}^{pt}\left[u\right]_{\beta, q}^{q\left(1-t\right)}.
\end{align}
Let $u\in \CW$ and taking into account that $1< q < p$, we can infer that
\begin{align*}
I\left(u\right)&\geq \dfrac{1}{p}\left[u\right]_{\alpha, p}^p+\dfrac{1}{q}\left[u\right]_{\beta, q}^q - \dfrac{\lambda}{p} C\left\|a\right\|_{L^{+\infty}\left(\R^N\right)}^{\frac{pt}{p_{\alpha}^*}}\left\|a\right\|_{L^{\left(\frac{q_{\beta}^*}{s}\right)^{\prime}}\left(\R^N\right)}^{p\left(1-t\right)}\left[u\right]_{\alpha, p}^{pt}\left[u\right]_{\beta, q}^{q\left(1 - t\right)}\\ &\geq
\dfrac{1}{p}\left[u\right]_{\alpha, p}^p - \dfrac{\lambda}{p} C\left\|a\right\|_{L^{+\infty}\left(\R^N\right)}^{\frac{pt}{p_{\alpha}^*}}\left\|a\right\|_{L^{\left(\frac{q_{\beta}^*}{s}\right)^{\prime}}\left(\R^N\right)}^{p\left(1-t\right)}\left[u\right]_{\alpha, p}^{pt}\left[u\right]_{\beta, q}^{q\left(1-t\right)}\\&
\dfrac{1}{p}\left[u\right]_{\alpha, p}^{pt}\left(\left[u\right]_{\alpha, p}^{p\left(1-t\right)} - \lambda C\left\|a\right\|_{L^{\infty}\left(\R^N\right)}^{\frac{pt}{p_{\alpha}^*}}\left\|a\right\|_{L^{\left(\frac{q_{\beta}^*}{s}\right)'}\left(\R^N\right)}^{p\left(1-t\right)}\left[u\right]_{\beta, q}^{q\left(1-t\right)}\right)
\end{align*}

Choosing $\left\|u\right\|_{\CW} = \mu$, by the definition $\CW = W^{\alpha, p}\left(\R^N\right)\bigcap W^{\beta, q}\left(\R^N\right)$, we have
\begin{align*}
\left\|u\right\|_{\CW}= \mu_1+ \mu_2.
\end{align*}
where $\left[u\right]_{\alpha, p}= \mu_1; \,\,\,\left[u\right]_{\beta, q}= \mu_2$.
Since $q<p$, this implies that $\left[u\right]_{\beta, q}^{q\left(1-t\right)}\leq \left[u\right]_{\beta, q}^{p\left(1-t\right)}$. Hence we have
\begin{align}\label{4.22}
I(u)\geq \dfrac{1}{p}\mu_1^{pt}\left(\mu_1^{p(1-t)} -\lambda C\left\|a\right\|_{L^{+\infty}\left(\R^N\right)}^{\frac{pt}{p_{\alpha}^*}}\left\|a\right\|_{L^{\left(\frac{q_{\beta}^*}{s}\right)'}\left(\R^N\right)}^{p\left(1-t\right)}\mu_2^{p(1-t)}\right).
\end{align}
For all $\varepsilon >0$, we set $\mu_2 = \varepsilon$ and $\mu_1 = \left(1+ \lambda C\left\|a\right\|_{L^{+\infty}\left(\R^N\right)}^{\frac{pt}{p_{\alpha}^*}}\left\|a\right\|_{L^{\left(\frac{q_{\beta}^*}{s}\right)'}\left(\R^N\right)}^{p\left(1-t\right)}\varepsilon^{p(1-t)}\right)^{\frac{1}{p(1-t)}}.$

This together with \eqref{4.22}  yields that
\begin{align*}
I(u)\geq \dfrac{1}{p}\left[1+ \lambda C\left\|a\right\|_{L^{+\infty}\left(\R^N\right)}^{\frac{pt}{p_{\alpha}^*}}\left\|a\right\|_{L^{\left(\frac{q_{\beta}^*}{s}\right)'}\left(\R^N\right)}^{p\left(1-t\right)}\varepsilon^{p(1-t)}\right]^{\frac{t}{1-t}}\geq \dfrac{1}{p^{\frac{t}{1-t}}}>0.
\end{align*}
By setting
 $\tau = 
\dfrac{1}{p^{\frac{t}{1-t}}}$ and $\left\|u\right\|_{\CW} 
= \dfrac{1}{p}\left[1+ \lambda C\left\|a\right\|
_{L^{\infty}\left(\R^N\right)}^{\frac{pt}{p_{\alpha}^*}}\left\|a\right\|
_{L^{\left(\frac{q_{\beta}^*}{s}\right)'}\left(\R^N
\right)}^{p\left(1-t\right)}\varepsilon^{q(1-t)}
\right]^{\frac{t}{1-t}}+ \varepsilon$,  we conclude that $I(u)\geq \tau>0$ for all $\left\|u
\right\|_{\CW} = \mu$.

$(2)$  Fix $\phi \in C_c^{\infty}\left(\R^N\right)$ such that $\phi \geq 0$ in $\R^N$ and $\phi\neq 0$. Since $1<q<p$, we may assume without loss of generality that $\phi$ is the normalized eigenfunction associated to the first eigenvalue $\lambda_1$ of the fractional $p-$Laplacian with weight $a$ satisfying $\displaystyle\int_{\R^N}\dfrac{a(x)}{2}\left|\varphi\right|^pdx = \dfrac{1}{p_1}$. We have that for all $t>0$,
\begin{align*}
I\left(t\phi\right)& = \dfrac{t^p}{p}\left[u\right]_{W^{\alpha, p}\left(\R^N\right)}^q + \dfrac{t^q}{q}\left[u\right]_{W^{\beta, q}\left(\R^N\right)}^q - \dfrac{\lambda t^p}{p}\displaystyle\int_{\R^N}a(x)\left|\phi\right|^pdx\\& = \dfrac{2 t^p}{p^2}\left(\lambda_1 - \lambda\right)+ \dfrac{t^q}{q}\left[\varphi\right]_{\beta, q}^q.
\end{align*}
Since $q<p$ and $\lambda_1<\lambda$, we deduce
\begin{align*}
\lim\limits_{t\to +\infty}I\left(t\phi\right) = \lim\limits_{t\to +\infty}t^p\left[\dfrac{2}{p^2}\left(\lambda_1 - \lambda\right) + \dfrac{t^{q-p}}{q}\left[\varphi\right]_{\beta, q}^q\right] = -\infty\,\,\,\text{in}\,\,\,\R^N.
\end{align*}

Thus we can choose $t_0>0$ and $u_0 =t_0\phi$  such that $I\left(u_0\right)<0$. This
fact shows that $(2)$ holds true.
\end{proof}
\begin{proof}[\bf Proof of Proposition \ref{proposition_4.2}]
By Lemma \ref{lem_4.4}, let
\begin{align*}
\Gamma = \left\{\gamma \in C\left(\left[0,1\right], \CW\right) \gamma(0) = 0, \gamma(1) = t_0\phi\right\}
\end{align*}
be the class of paths joining $0$ and $t_0\phi$. Setting
\begin{align*}
c: =  \inf\limits_{\gamma\in\Gamma}\max\limits_{\phi\in \gamma\left[0,1\right]}I\left(\phi(\gamma)\right).
\end{align*}
Since $I$ satisfies the Palais-Smale condition and Mountain pass geometry, by Lemma \ref{lem_4.3} and Lemma \ref{lem_4.4}, we can apply mountain pass theorem, Lemma \ref{lem_2.4}. Therefore, for any $\lambda> \lambda_1\left(\alpha, \beta, p, q\right)$, then $c$ is a critical value of $I$
associated to a critical point $u_0\in\CW$ . Namely, $I^{\prime}\left(u_0\right) = 0$ and $I\left(u_0\right) = c$. Since $c\geq \tau$, it follows from Lemma \ref{lem_4.4} that $u_0\neq 0$. The proof is now complete.
\end{proof}
Clearly, the proof of theorem \ref{theorem_1.3} is the direct consequence of the following propositions \ref{proposition_4.1}, \ref{proposition_4.2}.

\def\cprime{$'$} \def\polhk#1{\setbox0=\hbox{#1}{\ooalign{\hidewidth
  \lower1.5ex\hbox{`}\hidewidth\crcr\unhbox0}}}
  \def\cfac#1{\ifmmode\setbox7\hbox{$\accent"5E#1$}\else
  \setbox7\hbox{\accent"5E#1}\penalty 10000\relax\fi\raise 1\ht7
  \hbox{\lower1.15ex\hbox to 1\wd7{\hss\accent"13\hss}}\penalty 10000
  \hskip-1\wd7\penalty 10000\box7}


\end{document}